\documentclass[lettersize,onecolumn]{IEEEtran}
\usepackage{amsmath,amsfonts}
\usepackage{algorithmic}
\usepackage{algorithm}
\usepackage{array}
\usepackage[caption=false,font=normalsize,labelfont=sf,textfont=sf]{subfig}
\usepackage{textcomp}
\usepackage{stfloats}
\usepackage{url}
\usepackage{verbatim}
\usepackage{graphicx}
\usepackage{cite}
\usepackage{amssymb}
\usepackage{breqn}
 \usepackage{amsthm}
 \usepackage{amsmath}
 \usepackage{hyperref}
 \hypersetup{hidelinks}
 \usepackage{verbatim}
 \usepackage{tikz}
\usepackage{tikz-cd}
\begin{document}

\title{Explicit determination of a class of permutation rational functions in any characteristic}

\author{Yi Li, Deng Tang
\thanks{Yi Li, and Deng Tang are with the School of Computer Science, Shanghai Jiao Tong University, Shanghai 200240, China.(e-mail:
yi.li@sjtu.edu.cn; dengtang@sjtu.edu.cn).}
}

\markboth{Journal of \LaTeX\ Class Files,~Vol.~1, No.~2, December~2023}%
{Shell \MakeLowercase{\textit{et al.}}: A Sample Article Using IEEEtran.cls for IEEE Journals}

\IEEEpubid{0000--0000~\copyright~2023 IEEE}

\maketitle

\begin{abstract}
In this paper, we make use of the classification results of low-degree permutation rational functions together with their geometric properties to investigate rational functions that induce permutations on the multiplicative subgroup $\mu_{q+1}$, where $q$ is a prime power. By carefully analyzing the structural conditions under which such rational functions permute $\mu_{q+1}$, we obtain an explicit description of a broad class of permutation rational functions of small degree. As a direct application of these findings, we explicitly determine many permutation quadrinomials over $\mathbb{F}_{q^2}$ that are induced by degree-3 rational functions permuting $\mu_{q+1}$. Our approach not only unifies and extends several existing results in the literature but also provides a concrete geometric perspective for characterizing permutation polynomials over $\mathbb{F}_{q^2}$.
\end{abstract}

\newtheorem{exa}{Example}
\newtheorem{rmk}{Remark}
\newtheorem{defn}{Definition}
\newtheorem{pro}{Proposition}
\newtheorem{thm}{Theorem}
\newtheorem{lem}{Lemma}
\newtheorem{case}{Case}	
\newtheorem{cor}{Corollary}
\newtheorem{subcase}{Case}
\numberwithin{subcase}{case}
\section{Introduction}
\IEEEPARstart{L}{et} $q$ be a prime power and $\mathbb{F}_{q}$ be the finite field with $q$ elements. A polynomial $f\in \mathbb{F}_{q}[x]$ is called a permutation polynomial over $\mathbb{F}_q$ if the associated polynomial function $f:c\rightarrow f(c)$ from $\mathbb{F}_q$ into $\mathbb{F}_q$ is a permutation. In recent years, the study of permutation polynomials has expanded significantly because of their applications in coding theory \cite{OTCCFM2013,PPACC2007}, cryptography \cite{RSA1978,PKE1998}, and combinatorial designs \cite{AFSHDS2006}, and we refer the reader to \cite{PPOFFASORA2015,HFF2013,PPOFFAIA2019, wangDCC2024} and references therein for more details of the recent advances and contributions of the area. 

Numerous works have focused on constructing families of sparse permutation polynomials, while comparatively less attention has been given to the complete classification of permutation polynomials within specific classes. In 2015, Hou \cite{HOU201516} determined all permutation trinomials of $\mathbb{F}_{q^2}$ of the form $ax+bx^q+x^{2q-1}\in \mathbb{F}_{q^2}[x]$. Since then, a large portion of this research has focused on permutation polynomials over $\mathbb{F}_{q^2}$ of the form $f(x)=ax^{2q-1}+bx^{q}+cx+dx^{q^2-q+1}$. However, in the majority of existing results, the coefficients are constrained to belong to $\mathbb{F}_q$. In \cite{ding2025class}, Ding and Zieve determined all permutation quadrinomials of the form $f(x)=x+x^q+x^{2q-1}+dx^{q^2-q+1}$ over $\mathbb{F}_{q^2}$. Very recently, Ding et al. \cite{DXZCP} and  C.H. Chan et al. \cite{CHAN2026102734} completely determined the  permutation quadrinomials over $\mathbb{F}_{q^2}$ of the form $ax^{3q}+bx^{2q+1}+cx^{q+2}+dx^3$ up to linear equivalence. 

In this paper, by using the classification results of low-degree permutation rational functions, we explicitly determine all degree-$3$ rational functions of the form $\frac{d^qx^3+c^qx^2+b^qx+a^q}{ax^3+bx^2+cx+d}$ over $\mu_{q+1}$, where $a,b,c,d\in \mathbb{F}_{q^2}$, which induce a permutation on $\mu_{q+1}$. Our main theorems are stated as follows:

Let $q=p^m$, where $p$ is a prime and $m$ is a positive integer. For $a,b,c,d\in \mathbb{F}_{q^2}$, define 
\[D(x)=ax^{3}+bx^2+cx+dx, N(x)=d^qx^3+c^qx^2+b^qx+a^q,\]
\[g(x)=\frac{d^qx^3+c^qx^2+b^qx+a^q}{ax^3+bx^2+cx+d}.\]

When $g(x)$ has degree 1, we have

\begin{thm}\label{thm6}
     Let $q=p^m$, where $p$ is prime and $m$ is a positive integer. For $a,b,c,d\in \mathbb{F}_{q^2}$, the rational function  $g(x)=\frac{d^qx^3+c^qx^2+b^qx+a^q}{ax^3+bx^2+cx+d}$ permutes $\mu_{q+1}$ if and only if one of the following holds:

        (1) $d=0$, $c\neq 0$, $bc^q=ab^q$ and $c^{q+1}=a^{q+1}$;

    (2) $d=0$, $a=0$, and $c^{q+1}\neq b^{q+1}$;

    (3) $d\neq 0$, $d^qb\neq ac^q$, $d^qc\neq ab^q$, and $(\frac{d^qc-ab^q}{d^qb-ac^q})^q(\frac{d^{q+1}-a^{q+1}}{d^qb-ac^q})=\frac{d^qc-ab^q}{d^qb-ac^q}$;

    (4) $d\neq 0$, $d^qb\neq ac^q$, $d^qc=ab^q$, and $\frac{d^{q+1}-a^{q+1}}{d^qb-ac^q}=\frac{b^q}{c^q}\in \mu_{q+1}$.
    
\end{thm}

When $g(x)$ has degree $2$, we obtain
\begin{thm}\label{thm7}
     Let $q=2^m$, where $m$ is a positive integer. For $a,b,c,d\in \mathbb{F}_{q^2}$, the rational function  $g(x)=\frac{d^qx^3+c^qx^2+b^qx+a^q}{ax^3+bx^2+cx+d}$ permutes $\mu_{q+1}$ if and only if 

     (1) $c=d=0$, $a,b\neq 0$ and $\frac{a}{b}\in \mu_{q+1}$;

     (2) $c,d\neq 0$, $a=b=0$, $\frac{d}{c}\in \mu_{q+1}$;

     (3) $a,b,c,d\neq 0$, $\frac{c^q}{d^q}=\frac{b}{a}=\frac{d}{c}=\frac{a^q}{b^q}=\frac{d^{q+1}-a^{q+1}}{d^qc-ab^q}\in \mu_{q+1}$.
\end{thm}

When the characteristic of $\mathbb{F}_q$ is not $3$ and the degree of $g(x)$ is $3$, we have
\begin{thm}\label{thm8}
     Let $q=p^m$, where $p$ is prime and $m$ is a positive integer. For $a,b,c,d\in \mathbb{F}_{q^2}$, the rational function $g(x)=\frac{N(x)}{D(x)}=\frac{d^qx^3+c^qx^2+b^qx+a^q}{ax^3+bx^2+cx+d}\in \mathbb{F}_{q^2}(x)$. Assume that $\text{deg}(g)=3$. When the characteristic of $\mathbb{F}_q$ is $2$, the function $g(x)$ permutes $\mu_{q+1}$ if and only if Proposition \ref{pro1} holds and either
    
    (1) $\text{gcd}(3,q-1)=1$ and $(ad+bc)x^2+(a^{q+1}+b^{q+1}+c^{q+1}+d^{q+1})x+(a^qd^q+b^qc^q)=0$ has a root in $\mu_{q+1}$, or 
    
    (2) $\text{gcd}(3,q+1)=1$ and $(ad+bc)x^2+(a^{q+1}+b^{q+1}+c^{q+1}+d^{q+1})x+(a^qd^q+b^qc^q)=0$ has no roots in $\mu_{q+1}$. 
    
    When the characteristic of $\mathbb{F}_q$ is neither $2$ or $3$, the function $g(x)$ permutes $\mu_{q+1}$ if and only if Proposition \ref{pro1} holds and either 
    
    (1) $\text{gcd}(3,q-1)=1$ and  $(3ac-b^2)x^2+(2bc^q-3d^qc-3ab^q)x+3d^qb^q-c^{2q}=0$ has a root in $\mu_{q+1}$, or 
    
    (2) $\text{gcd}(3,q+1)=1$ and $ (3ac-b^2)x^2+(2bc^q-3d^qc-3ab^q)x+3d^qb^q-c^{2q}=0$ has no roots in $\mu_{q+1}$.
\end{thm}
It is easy to see that when $g(x)$ is not separable, its permutation behavior of on $\mu_{q+1}$ can be readily determined (see Section 4.3 case 2). When the characteristic of $\mathbb{F}_q$ is $3$ and $g(x)$ is a separable rational function of degree $3$, we obtain

\begin{thm}\label{thm9}
    Let $q=p^m$, where $p$ is prime and $m$ is a positive integer. For $a,b,c,d\in \mathbb{F}_{q^2}$, the rational function $g(x)=\frac{N(x)}{D(x)}=\frac{d^qx^3+c^qx^2+b^qx+a^q}{ax^3+bx^2+cx+d}\in \mathbb{F}_{q^2}(x)$. Assume that $\text{deg}(g)=3$. Let $\alpha=\frac{cd^q-ab^q}{bd^q-ac^q}$ and $\beta=\frac{c^q}{b}$. When the characteristic of $\mathbb{F}_q$ is $3$, the function $g(x)$ permutes $\mu_{q+1}$ if and only if $\alpha, \beta\in \mu_{q+1}$,  Proposition \ref{pro2} holds, and one of the following conditions is satisfied:

    (1) $\alpha\not\in \mathbb{F}_q$, $\beta\not\in \mathbb{F}_q$, and $\frac{B_3-\beta A_3}{\beta A_1-B_1}$ is a nonsquare of $\mathbb{F}_q^{*}$, where 
    $A_1=d^q\alpha^3+c^q\alpha^2+b^q\alpha+a^q$, $A_3=c^q(2\alpha^2+1)+b^q(2\alpha+\alpha^3)$, $B_1=a\alpha^3+b\alpha^2+c\alpha+d$, and $B_3=b(2\alpha^2+1)+c(2\alpha+\alpha^3)$;

    (2) $\alpha\not\in \mathbb{F}_q$, $\beta\in \mathbb{F}_q$, and $\frac{\beta\alpha^qB_3-\alpha A_3}{\alpha A_1-\beta\alpha^q B_1}$ is a nonsquare of $\mathbb{F}_q^{*}$, where $A_1,A_3,B_1,B_3$ are defined as in $(1)$;

    (3) $\alpha\in \mathbb{F}_q$, $\beta\not\in \mathbb{F}_q$, and $\frac{B_3-\beta A_3}{\beta A_1-B_1}$ is a nonsquare of $\mathbb{F}_q^{*}$, where $A_1=d^q\alpha^3+c^q\alpha^2+b^q\alpha+a^q$, $A_3=c^q(2\alpha^2+\alpha^2\beta^{q-1})+b^q(2\alpha+\alpha\beta^2)$, $B_1=a\alpha^3+b\alpha^2+c\alpha+d$, and $B_3=b(2\alpha^2+\alpha^2\beta^{q-1})+c(2\alpha+\alpha\beta^2)$;

    (4) $\alpha\in \mathbb{F}_q$, $\beta\in \mathbb{F}_q$. Let $\gamma\in \mathbb{F}_{q^2}$ be an element of order $q+1$. Then the expression $\frac{\beta\gamma^qB_3-\gamma A_3}{\gamma A_1-\beta\gamma^q B_1}$ is a nonsquare of $\mathbb{F}_q^{*}$, where $A_1=d^q\alpha^3+c^q\alpha^2+b^q\alpha+a^q$, $A_3=c^q(2\alpha^2+\alpha^2\gamma^{q-1})+b^q(2\alpha+\alpha\gamma^2)$, $B_1=a\alpha^3+b\alpha^2+c\alpha+d$, and $B_3=b(2\alpha^2+\alpha^2\gamma^{q-1})+c(2\alpha+\alpha\gamma^2)$.
\end{thm}
Using these results and together with those in Section $3$, we can explicitly determine many permutation polynomials induced from degree-3 rational functions which permute $\mu_{q+1}$, including  permutation polynomials of the form $ax^{2q-1}+bx^q+cx+dx^{q^2-q+1}$ and $ax^{3q}+bx^{2q+1}+cx^{q+2}+dx^3$ over $\mathbb{F}_{q^2}$, where $a,b,c,d\in\mathbb{F}_{q^2}$. These results are presented in Section 5.

The rest of the paper is organized as follows. Section 2 introduces some preliminary results and  techniques. Section 3 explicitly determines the degree of $g(x)$ in terms of coefficients. Section 4 provides conditions under which $g(x)$ permutes $\mu_{q+1}$. In Section 5, based on these conditions, we completely determine several classes of permutation polynomials, including those of the forms $ax^{2q-1}+bx^q+cx+dx^{q^2-q+1}$ and $ax^{3q}+bx^{2q+1}+cx^{q+2}+dx^3$ over $\mathbb{F}_{q^2}$. Section 6 concludes the paper and discusses future work.

\section{Preliminaries}
In this section, we present some known results to be used in our proof. Most of these results are from \cite{ding2023determination}.

There are several criteria to decide whether a polynomial is a permutation polynomial. The AGW criterion is one of the most useful criteria in the literature.

\begin{lem}\cite{akbary2011constructing}(The AGW criterion)
      Let $A,S$ and $\bar{S}$ be finite sets with $|S|=|\bar{S}|$. Let $f, \bar{f}, \lambda, \bar{\lambda}$ be maps on finite sets such that $f: A\rightarrow A$, $\bar{f}:S\rightarrow \bar{S}$, $\lambda: A\rightarrow S$, $\bar{\lambda}: A\rightarrow \bar{S}$, and $\bar{\lambda}\circ f=\bar{f}\circ \lambda$. The relation can be illustrated in a commutative diagram.
      
\begin{center}
    \begin{tikzcd}
A \arrow[rr, "f"] \arrow[dd, "\lambda"] &  & A \arrow[dd, "\bar{\lambda}"] \\
                                        &  &                               \\
S \arrow[rr, "\bar{f}"]                 &  & \bar{S}                      
\end{tikzcd}
\end{center}

If $\lambda$ and $\bar{\lambda}$ are surjective, then the following are equivalent:

(i) $f$ is a bijection from $A$ to itself;

(ii) $\bar{f}$ is a bijection from $S$ to $\bar{S}$ and $f$ is injective on $\lambda^{-1}(s)$ for each $s\in S$.
\end{lem}

In particular, if we let $A=\mathbb{F}_{q}^{*}$ and the maps $\lambda=\bar{\lambda}=x^s$, where $q-1=ds$ for some positive integers $s,d$, then we obtain the so-called multiplicative case of the AGW criterion.

\begin{cor}\cite{lee1997some,wan1991permutation,wang2007cyclotomic,zieve2009some}\label{cor1}
     Let $q-1=ds$, where $d$ and $s$ are two positive integers, $q$ is a prime power and $f(x)$ is an arbitrary polynomial over $\mathbb{F}_q$. Then $p(x)=x^rf(x^s)$ is a permutation polynomial over $\mathbb{F}_{q}$ if and only if $\text{gcd}(r,s)=1$ and $x^rf(x)^s$ permutes the set $\mu_{d}$ of $d$-th roots of unity.
\end{cor}

Next, we provide some background on self-conjugate reciprocal polynomials. Before introducing their definition, we first fix some notation.

(1) Let $g(x)\in \mathbb{F}_{q^2}(x)$, we define $g^{(q)}(x)$ to be the rational function obtained from $g(x)$.

(2) For any nonzero $A(x)\in \mathbb{F}_{q^2}[x]$, we define $\hat{A}(x):=x^{\text{deg}(A)}A^{(q)}(1/x)$.

Explicitly, if $A(x)=\sum\limits_{i=0}^{n}\alpha_i x^i$ with $\alpha_i\in \mathbb{F}_{q^2}$ and $\alpha_n\neq 0$ then $\hat{A}(x)=\sum\limits_{i=0}^{n}\alpha_i^q x^{n-i}$.

Now, we present the definition of self-conjugate reciprocal polynomials.

\begin{defn}\cite{ding2023determination}
    \hspace{0.5em}We say that a nonzero $A(x)\in \mathbb{F}_{q^2}[x]$ is self-conjugate reciprocal (or SCR for short) if $\hat{A}(x)=\alpha A(x)$ for some $\alpha\in \mathbb{F}_{q^2}$.
\end{defn}

SCR polynomials have many simple properties.

\begin{lem}\cite{ding2023determination}\label{lem2}
    All of the following hold:

    (1) If $A(x)\in \mathbb{F}_{q^2}[x]$ is SCR, then $\hat{A}(x)/A(x)\in \mu_{q+1}$.

    (2) For nonconstant $g_1,g_2\in \mathbb{F}_{q^2}(x)$, we have $(g_1\circ g_2)^{(q)}=g_1^{(q)}\circ g_2^{(q)}$.

    (3) If $A(x)\in \mathbb{F}_{q^2}[x]$ is nonzero and $\alpha\in \bar{\mathbb{F}}_q^{*}$, then the multiplicity of $\alpha$ as a root of $A(x)$ equals the multiplicity of $\alpha^{-q}$ as a root of $\hat{A}(x)$.

    (4) $A(x)\in \mathbb{F}_{q^2}[x]$ is SCR if and only if the multiset of roots of $A(x)$ is preserved by the function $\alpha\mapsto \alpha^{-q}$.

    (5) Every degree-1 SCR polynomial has a root in $\mu_{q+1}$.
\end{lem}

In order to prove $x^rf(x)^s$ permutes the set $\mu_d$ for specific values of $s$ and $d$, we require some knowledge of rational functions.

Let $K$ be a field, $g(x)=\frac{N(x)}{D(x)}$, where $N(x), D(x)\in K[x]$ with $N(x)$ and $D(x)$ are coprime. Hence $g(x)$ can be viewed as a function $\mathbb{P}^{1}(K)\rightarrow \mathbb{P}^{1}(K)$, where $\mathbb{P}^{1}(K)=K\cup \{\infty \}$. When $g(x)\neq 0$, we can define the degree of $g(x)$ as $\text{deg}(g)=\text{max}(\text{deg}(N),\text{deg}(D))$. We say that a nonconstant $g(x)\in K(x)$ is separable if the field extension $K(x)/K(g(x))$ is separable, where $x$ is transcendental over $K$. It is known that $g(x)$ is separable if and only if $g(x)\not\in K(x^p)$, where $p$ is the characteristic of $K$. For a proof, see \cite[Lemma 2.2]{ding2020low}.

We say that nonconstant $f,g\in K(x)$ are linearly equivalent if $g=\rho\circ f\circ \sigma$ for some $\rho,\sigma\in K(x)$ of degree $1$. 

Rational functions of degree 1 possess many simple properties. 

For any field $K$ and any $\rho(x)\in K(x)$ of degree $1$, if we write $\rho(x)=(\alpha x+\beta)/(\gamma x+\delta)$ with $\alpha, \beta, \gamma, \delta\in K$, then its inverse is $\rho^{-1}(x)=(\delta x-\beta)/(-\gamma x+\alpha)$. If we view $\rho(x)$ as an element of $PGL_2(K)$, with $A=\begin{bmatrix}
\alpha & \beta \\
\gamma & \delta
\end{bmatrix}$, then it is easy to see that $\rho^{-1}(x)$ corresponds to $A^{-1}=\begin{bmatrix}
\delta & -\beta \\
-\gamma & \alpha
\end{bmatrix}$. Therefore, we can regard rational functions of degree 1 as elements of the projective linear group $PGL_2(K)$. When $K=\mathbb{F}_{q^2}$, $\rho\in \mathbb{F}_{q^2}(x)$ has degree $1$, then $(\rho^{(q)})^{-1}=(\rho^{-1})^{(q)}$.

The following two lemmas are very useful in studying the bijection of $\mu_{q+1}$ and $\mathbb{P}^{1}(\mathbb{F}_q)$.

\begin{lem}\cite{zieve2013permutation}\label{lem3}
    A degree-1 $\rho(x)\in \mathbb{F}_{q^2}(x)$ permutes $\mu_{q+1}$ if and only if $\rho(x)=(\beta^qx+\alpha^q)/(\alpha x+\beta)$ for some $\alpha, \beta\in \mathbb{F}_{q^2}$ with $\alpha^{q+1}\neq \beta^{q+1}$.
\end{lem}

\begin{lem}\cite{zieve2013permutation}\label{lem4}
    A degree-1 $\rho(x)\in \mathbb{F}_{q^2}(x)$ maps $\mu_{q+1}$ to $\mathbb{P}^{1}(\mathbb{F}_q)$ if and only if $\rho(x)=(\delta x+\gamma \delta^q)/(x+\gamma)$ for some $\gamma\in \mu_{q+1}$ and $\delta\in \mathbb{F}_{q^2}\setminus \mathbb{F}_q$.
\end{lem}

The following two lemmas describe the relationship between $\mathbb{F}_{q^2}$, $\mu_{q+1}$, and $\mathbb{P}^{1}(\mathbb{F}_q)$.

\begin{lem}\cite{zieve2013permutation}\label{lem5}
Write $g_0(x)=x^rA(x)^{q-1}$ where $r$ is an integer, $q$ is a prime power, and $A(x)\in \mathbb{F}_{q^2}[x]$ is nonzero. Then $g_0(x)$ maps $\mu_{q+1}$ into $\mu_{q+1}\cup \{0\}$, and if $A(x)$ has no roots in $\mu_{q+1}$, then $g_0(x)$ induces the same function on $\mu_{q+1}$ as does $g(x)=x^sA^{(q)}(1/x)/A(x)$, for any integer $s$ with $r\equiv s\mod (q+1)$. In particular, $g_0(x)$ permutes $\mu_{q+1}$ if and only if $A(x)$ has no roots in $\mu_{q+1}$ and $g(x)$ permutes $\mu_{q+1}$.
\end{lem}

\begin{lem}\cite{zieve2013permutation}\label{lem6}
    Let $g(x)\in \mathbb{F}_{q^2}(x)$ be a nonconstant rational function having the form $g(x)=x^sA^{(q)}(1/x)/A(x)$ where $s$ is an integer, $q$ is a prime power, and $A(x)\in \mathbb{F}_{q^2}[x]$. Let $h(x):=\rho\circ g\circ\sigma^{-1}$ where $\rho,\sigma\in \mathbb{F}_{q^2}(x)$ are degree-$1$ rational functions that map $\mu_{q+1}$ to $\mathbb{P}^{1}(\mathbb{F}_q)$. Then $h(x)$ is in $\mathbb{F}_q(x)$, and $h(x)$ permutes $\mathbb{P}^{1}(\mathbb{F}_q)$ if and only if $g(x)$ permutes $\mu_{q+1}$.
\end{lem}

 Here we introduce the concepts of branch points and ramification, which are key tools to describe the geometric properties of a rational function.

 For any nonconstant $g(x)\in \bar{\mathbb{F}}_q(x)$, and any $\alpha\in \mathbb{P}^{1}(\bar{\mathbb{F}}_q)$, the ramification index $e_g(\alpha)$ is the multiplicity of $\alpha$ as a $g$-preimage of $g(\alpha)$. Explicitly, for any degree-$1$, $\rho,\sigma\in \bar{\mathbb{F}}_q(x)$ such that $\sigma(0)=\alpha$, and $\rho(g(\alpha))=0$, the positive integer $e_g(\alpha)$ is the degree of the lowest-degree term of the numerator of $\rho\circ g\circ \sigma$. 

 For $\beta\in \mathbb{P}^{1}(\bar{\mathbb{F}}_q)$, we define the ramification multiset of $g(x)$ over $\beta$ to be the multiset $E_g(\beta)$ consisting of the ramification indices $e_g(\alpha)$ with $\alpha\in g^{-1}(\beta)$. It can be seen that $E_g(\beta)$ is a collection of positive integers whose sum is $\text{deg}(g)$. If $g(x)$ is constant, then $E_g(\beta)$ is the empty multiset.

 We say that $\alpha\in \mathbb{P}^{1}(\bar{\mathbb{F}}_q)$ is a ramification point of $g(x)$ if $e_g(\alpha)>1$. We say that $\beta\in \mathbb{P}^{1}(\bar{\mathbb{F}}_q)$ is a branch point of $g(x)$ if $\beta=g(\alpha)$ for some ramification point $\alpha$ of $g(x)$. 

The following lemma provides a classification of separable permutation rational functions of degree $3$.

\begin{lem}\cite{ding2020low}\label{lem7}
    A separable degree $3$ rational function $g(x)\in \mathbb{F}_q(x)$ permutes $\mathbb{P}^{1}(\mathbb{F}_q)$ if and only if one of the following holds:

    (1) $q\equiv 2 \mod 3$ and $g(x)$ is linearly equivalent to $x^3$;

    (2) $q\equiv 1\mod 3$ and $g(x)=\rho\circ x^3\circ \sigma^{-1}$ for some degree-$1$ $\rho,\sigma\in \mathbb{F}_{q^2}(x)$ which map $\mu_{q+1}$ to $\mathbb{P}^{1}(\mathbb{F}_q)$;

    (3) $q\equiv 0\mod 3$ and $g(x)$ is linearly equivalent to $x^3-\alpha x$ for some nonsquare $\alpha\in \mathbb{F}_q^{*}$.
\end{lem}

This lemma yields the following corollary:

\begin{cor}\cite{ding2023determination}
    Every degree-$3$ permutation rational function $h(x)\in \mathbb{F}_q(x)$ has ramification multiset $[3]$ over each of its branch point.
\end{cor}

\section{The degree of $g(x)$}
From the introduction, Corollary \ref{cor1}, and Lemma \ref{lem5}, when $f(x)=ax^3+bx^2+cx+d$, where $a,b,c,d\in \mathbb{F}_{q^2}$ and $r\equiv 3 \mod (q+1)$, to prove $x^rf(x^{q-1})$ is a permutation polynomial over $\mathbb{F}_{q^2}$, it is sufficient to show  $g(x)=\frac{N(x)}{D(x)}=\frac{d^qx^3+c^qx^2+b^qx+a^q}{ax^3+bx^2+cx+d}$ permutes $\mu_{q+1}$ and that the denominator $ax^3+bx^2+cx+d$ has no roots in $\mu_{q+1}$.  Before proving when $g(x)$ permutes $\mu_{q+1}$, in this section, we first characterize, case by case, when the degree of $g(x)$ is $0,1,2,$ or $3$. We denote by $H(x)$ the greatest common divisor of $N(x)$ and $D(x)$.

\begin{case}[$d=0$]
In this case, $D(x)=ax^3+bx^2+cx$ and $N(x)=c^qx^2+b^qx+a^q$. If $c=0$, then $N(x)=b^qx+a^q$ and $D(x)=ax^3+bx^2$. It is easy to see that when $b=0$, the degree of $g(x)$ is $3$, provided that $a\neq 0$. Next, we consider $b\neq 0$. It is easy to know when $\frac{a^q}{b^q}\in \mu_{q+1}$, $\text{deg}(H)=1$. Otherwise, $\text{deg}(g)=3$ if $a\neq 0$; $\text{deg}(g)=1$ if $a=0$. If $c\neq 0$ and $b=0$, then $D(x)=ax^3+cx$ and $N(x)=c^qx^2+a^q$. In this case, $H(x)|(c^{q+1}-a^{q+1})$. If $c^{q+1}\neq a^{q+1}$, then $\text{deg}(g)=3$ if $a\neq 0$; $\text{deg}(g)=1$ if $a=0$. If $c^{q+1}=a^{q+1}$, then $H(x)=ax^2+c$. The degree of $g$ is 1. If $c\neq 0$ and $b\neq 0$, then $D(x)=ax^3+bx^2+cx$ and $N(x)=c^qx^2+b^qx+a^q$. In this case, $H(x)|(bc^q-ab^q)x+(c^{q+1}-a^{q+1})$. If $bc^q=ab^q$ and $c^{q+1}=a^{q+1}$, then $H(x)=c^qx^2+b^qx+a^q$ and $\text{deg}(g)=1$. If $bc^q=ab^q$ and $c^{q+1}\neq a^{q+1}$, then $H(x)=1$ and $\text{deg}(g)=3$. If $bc^q\neq ab^{q}$ and $\frac{a^{q+1}-c^{q+1}}{bc^q-ab^q}\in \mu_{q+1}$, then $H(x)=(bc^q-ab^q)x+(c^{q+1}-a^{q+1})$, and $\text{deg}(g)=2$ if $a\neq 0$, and $\text{deg}(g)=0$ if $a=0$. Otherwise, $H(x)=1$ and $\text{deg}(g)=3$ if $a\neq 0$; $\text{deg}(g)=1$ if $a=0$.
\end{case}

Next we consider $d\neq 0$. In this case, $D(x)=ax^3+bx^2+cx+d$ and $N(x)=d^qx^3+c^qx^2+b^qx+a^q$. It is easy to know $H(x)|d^qD(x)-aN(x)=(d^qb-ac^q)x^2+(d^qc-ab^q)x+(d^{q+1}-a^{q+1})$.

\begin{case}[$d\neq 0$, $d^qb=ac^q$]
    In this case, $H(x)|(d^qc-ab^q)x+(d^{q+1}-a^{q+1})$. If $d^qc=ab^q$ and $d^{q+1}=a^{q+1}$, then $H(x)=ax^3+bx^2+cx+d$, which means $g(x)$ is constant on $\mu_{q+1}$. If $d^qc=ab^q$ and $d^{q+1}\neq a^{q+1}$, then $H(x)=1$ and $\text{deg}(g)=3$. If $d^qc\neq ab^q$ and $d^{q+1}=a^{q+1}$, then $H(x)=1$ and $\text{deg}(g)=3$. We have a lemma to deal with the case $d^qc\neq ab^q$ and $d^{q+1}\neq a^{q+1}$.

    \begin{lem}\label{lem8}
        If $d\neq 0$, $d^qb=ac^q$, $d^qc\neq ab^q$, and $d^{q+1}\neq a^{q+1}$, then $\text{deg}(H)=1$ if and only if $(d^{q+1}-a^{q+1})^{q+1}=(d^qc-ab^q)(dc^q-a^qb)$.
    \end{lem}

    \begin{proof}
        The necessity is trivial. Since $\text{deg}(H)=1$ and $H(x)|(d^qc-ab^q)x+(d^{q+1}-a^{q+1})$, it implies $(d^{q+1}-a^{q+1})^{q+1}=(d^qc-ab^q)(dc^q-a^qb)$.

        Now assume $(d^{q+1}-a^{q+1})^{q+1}=(d^qc-ab^q)(dc^q-a^qb)$, which means the root of $(d^qc-ab^q)x+(d^{q+1}-a^{q+1})$ is in $\mu_{q+1}$. Plugging it into $N(x)$, we have
        \begin{eqnarray}\label{1}
            d^q\left(\frac{a^{q+1}-d^{q+1}}{d^qc-ab^q}\right)^3+c^q\left(\frac{a^{q+1}-d^{q+1}}{d^qc-ab^q}\right)^2+b^q\left(\frac{a^{q+1}-d^{q+1}}{d^qc-ab^q}\right)+a^q.
        \end{eqnarray}

        Simplify it, we have 
        \begin{dmath}\label{2}
            d^q(a^{q+1}-d^{q+1})^3+c^q(a^{q+1}-d^{q+1})^2(d^qc-ab^q)+b^q(a^{q+1}-d^{q+1})(d^qc-ab^q)^2+a^q(d^qc-ab^q)^3.
        \end{dmath}
         Note that $(d^{q+1}-a^{q+1})^{q+1}=(d^{q+1}-a^{q+1})^2$. Using $(a^{q+1}-d^{q+1})^2=(d^qc-ab^q)$$(dc^q-a^qb)$, we have
        \begin{dmath}\label{3}
            d^q(a^{q+1}-d^{q+1})(d^qc-ab^q)(dc^q-a^qb)+c^q(d^qc-ab^q)^2(dc^q-a^qb)+b^q(a^{q+1}-d^{q+1})(d^qc-ab^q)^2+a^q(d^qc-ab^q)^3.
        \end{dmath}
        Since $d^qc\neq ab^q$, divide it by $(d^qc-ab^q)$, we have
        \begin{dmath}\label{4}
             d^q(a^{q+1}-d^{q+1})(dc^q-a^qb)+c^q(d^qc-ab^q)(dc^q-a^qb)+b^q(a^{q+1}-d^{q+1})(d^qc-ab^q)+a^q(d^qc-ab^q)^2.
        \end{dmath}
        Using $d^qb=ac^q$, we have 
        \begin{eqnarray}\label{5}
            -c^q(a^{q+1}-d^{q+1})^2+(d^qc-ab^q)(-b^qd^{q+1}+dc^{2q}-a^qc^qb+a^qd^qc).
        \end{eqnarray}

        Note that $(d^qc-ab^q)(dc^{2q}-a^qc^qb)=c^q(a^{q+1}-d^{q+1})^2$, Equation (\ref{5}) is transformed into
        \begin{eqnarray}\label{6}
            (d^qc-ab^q)(-b^qd^{q+1}+a^qd^qc)=(d^qc-ab^q)(d^q(-b^qd+a^qc))=0.
        \end{eqnarray}

        Equation (\ref{6}) shows $\frac{a^{q+1}-d^{q+1}}{d^qc-ab^q}$ is a root of $N(x)$. From Lemma \ref{lem2}, if $\alpha$ is a root of $N(x)$, then $\alpha^{-q}$ is a root of $D(x)$. Since $\alpha=\frac{a^{q+1}-d^{q+1}}{d^qc-ab^q}\in \mu_{q+1}$, we have $\alpha^{-q}=\alpha$. Therefore, $\alpha$ is also a root of $D(x)$, meaning that $\alpha$ is a common root of $N(x)$ and $D(x)$. Hence $\text{deg}(H)=1$.
    \end{proof}

    Using Lemma \ref{lem8}, we know that if $(d^{q+1}-a^{q+1})^{q+1}=(d^qc-ab^q)(dc^q-a^qb)$, then $\text{deg}(H)=1$; otherwise, $H(x)=1$.
\end{case}

\begin{case}[$d\neq 0$, $d^qb\neq ac^q$]
    In this case, $\text{deg}(H)$ can be $0,1,$ or $2$. We present the following lemma to characterize when $\text{deg}(H)=2$.

    \begin{lem}\label{lem9}
        Under the conditions $d\neq 0$ and $d^qb\neq ac^q$, $\text{deg}(H)=2$ if and only if one of the following statements holds:

        (i) $d^qc\neq ab^q$, $\left(\frac{d^qc-ab^q}{d^qb-ac^q}\right)^q\frac{d^{q+1}-a^{q+1}}{d^qb-ac^q}=\frac{d^qc-ab^q}{d^qb-ac^q}$;

        (ii) $d^qc=ab^q$, $\frac{d^{q+1}-a^{q+1}}{d^qb-ac^q}=\frac{b^q}{d^q}\in \mu_{q+1}$.

    \end{lem}

    \begin{proof}
        Assume that $\text{deg}(H)=2$. Since $\text{deg}(d^qD(x)-aN(x))=2$, the polynomials $H(x)$ and $d^qD(x)-aN(x)$ must have the same root. Assume that in the splitting field, the three roots of $D(x)$ are $\alpha_1,\alpha_2,\alpha_3$. Since $N(x)=x^3D^{(q)}(1/x)$, the three roots of $N(x)$ in the splitting field are $\alpha_1^{-q}$, $\alpha_2^{-q}$, $\alpha_3^{q}$. Since $H(x)$ is the greatest common divisor of $D(x)$ and $N(x)$, the roots of $H(x)$ in the splitting field are $\{\alpha_i,\alpha_j\}$ or $\{\alpha_i^{-q},\alpha_j^{-q}\}$, where $i,j\in \{1,2,3\}$. Hence we know the multiset of roots of $H(x)$ is preserved by the function $\alpha \mapsto \alpha^{-q}$, which tells us $H(x)$ is an SCR polynomial from Lemma \ref{lem2}. Using Lemma \ref{lem10} below, we know $H(x)$ is an SCR polynomial if and only if $\left(\frac{d^qc-ab^q}{d^qb-ac^q}\right)^q\frac{d^{q+1}-a^{q+1}}{d^qb-ac^q}=\frac{d^qc-ab^q}{d^qb-ac^q}$ and $\left(\frac{d^{q+1}-a^{q+1}}{d^qb-ac^q}\right)^{q+1}=1$. If $d^qc\neq ab^q$, we can use $\left(\frac{d^qc-ab^q}{d^qb-ac^q}\right)^q\frac{d^{q+1}-a^{q+1}}{d^qb-ac^q}=\frac{d^qc-ab^q}{d^qb-ac^q}$ to get  $\left(\frac{d^{q+1}-a^{q+1}}{d^qb-ac^q}\right)^{q+1}=1$. Hence the condition $\left(\frac{d^{q+1}-a^{q+1}}{d^qb-ac^q}\right)^{q+1}=1$ is redundant. When $d^qc=ab^q$, the condition $\left(\frac{d^qc-ab^q}{d^qb-ac^q}\right)^q\frac{d^{q+1}-a^{q+1}}{d^qb-ac^q}=\frac{d^qc-ab^q}{d^qb-ac^q}$ is obviously true. Since $x^2+\frac{d^{q+1}-a^{q+1}}{d^qb-ac^q}$ is a factor of $d^qx^3+c^qx^2+b^qx+a^q$. Using the factorization, we can get $\frac{d^{q+1}-a^{q+1}}{d^qb-ac^q}=\frac{b^q}{d^q}\in \mu_{q+1}$.

        Conversely, if condition (ii) is satisfied, then $x^2+\frac{d^{q+1}-a^{q+1}}{d^qb-ac^q}$ is a factor of $d^qx^3+c^qx^2+b^qx+a^q$ and $ax^3+bx^2+cx+d$. Hence $\text{deg}(H)=2$. If condition (i) is satisfied, we know $d^qD(x)-aN(x)$ is an SCR polynomial from Lemma \ref{lem10}. We only need to show $x^2+\frac{d^qc-ab^q}{d^qb-ac^q}x+\frac{d^{q+1}-a^{q+1}}{d^qb-ac^q}$ is a factor of $N(x)=d^qx^3+c^qx^2+b^qx+a^q$. Assume $N(x)$ has a factorization $d^q(x^3+\frac{c^q}{d^q}x^2+\frac{b^q}{d^q}x+\frac{a^q}{d^q})=d^q\left(x^2+\frac{d^qc-ab^q}{d^qb-ac^q}x+\frac{d^{q+1}-a^{q+1}}{d^qb-ac^q}\right)\left(x+\alpha\right)$, where $\alpha$ is a root of $N(x)$ and not a root of $d^qD(x)-aN(x)$. Then by comparing coefficients, we have 
        \begin{eqnarray}\label{eqn7}
            \left\{
            \begin{array}{cc}
                 \frac{d^qc-ab^q}{d^qb-ac^q}+\alpha&=\frac{c^q}{d^q}  \\
               \alpha\frac{d^qc-ab^q}{d^qb-ac^q}+\frac{d^{q+1}-a^{q+1}}{d^qb-ac^q} &=\frac{b^q}{d^q}\\
                \alpha\frac{d^{q+1}-a^{q+1}}{d^qb-ac^q}&=\frac{a^q}{d^q}.
            \end{array}
            \right.
        \end{eqnarray}

        From the third equation of (\ref{eqn7}), we have
        \begin{eqnarray}\label{eqn8}
            \alpha=\frac{a^q}{d^q}\frac{d^qb-ac^q}{d^{q+1}-a^{q+1}}.
        \end{eqnarray}

        To show $\alpha$ is a solution of Equation (\ref{eqn7}), we show $\alpha=\frac{a^q}{d^q}\frac{d^qb-ac^q}{d^{q+1}-a^{q+1}}$ also satisfies the first and second equations of (\ref{eqn7}). Plugging it into the first equation of (\ref{eqn7}), we have
        \begin{eqnarray}\label{eqn9}
            \frac{d^qc-ab^q}{d^qb-ac^q}+\frac{a^q(d^qb-ac^q)}{d^q(d^{q+1}-a^{q+1})}=\frac{c^q}{d^q},
        \end{eqnarray}
        which is equivalent to
        \begin{eqnarray}\label{eqn10}
            d^q(d^{q+1}-a^{q+1})(d^qc-ab^q)+a^q(d^qb-ac^q)^2=c^q(d^qb-ac^q)(d^{q+1}-a^{q+1}).
        \end{eqnarray}

        Using the identity $(d^qc-ab^q)(d^{q+1}-a^{q+1})=(dc^q-a^qb)(d^qb-ac^q)$, we have
        \begin{dmath}\label{eqn11}
            d^q(dc^q-a^qb)(d^qb-ac^q)+a^q(d^qb-ac^q)^2=(d^qb-ac^q)(d^{q+1}c^q-a^{q+1}c^q)=c^q(d^qb-ac^q)(d^{q+1}-a^{q+1}),
        \end{dmath}
        which is the right side of Equation (\ref{eqn10}). Hence $\alpha$ satisfies the first equation of (\ref{eqn7}).  

        Now we show $\alpha$ satisfies the second equation of (\ref{eqn7}). Plugging it into the second equation of (\ref{eqn7}), we have
        \begin{eqnarray}\label{eqn12}
            \frac{a^q}{d^q}\left(\frac{d^qc-ab^q}{d^{q+1}-a^{q+1}}\right)+\frac{d^{q+1}-a^{q+1}}{d^qb-ac^q}=\frac{b^q}{d^q},
        \end{eqnarray}
        which is equivalent to
        \begin{eqnarray}\label{eqn13}
            a^q(d^qc-ab^q)(d^qb-ac^q)+d^q(d^{q+1}-a^{q+1})^2=b^q(d^{q+1}-a^{q+1})(d^qb-ac^q).
        \end{eqnarray}

        Using the identity $(d^qc-ab^q)(d^{q+1}-a^{q+1})=(dc^q-a^qb)(d^qb-ac^q)$ and its Frobenius $(dc^q-a^qb)(d^{q+1}-a^{q+1})=(d^qc-ab^q)(db^q-a^qc)$, we have
        \begin{eqnarray}\label{eqn14}
            (d^{q+1}-a^{q+1})^2=(db^q-a^qc)(d^qb-ac^q).
        \end{eqnarray}

        Using Equation (\ref{eqn14}), the left side of Equation (\ref{eqn13}) is transformed into
        \begin{eqnarray}\label{eqn15}
            a^q(d^qc-ab^q)(d^qb-ac^q)+d^q(db^q-a^qc)(d^qb-ac^q)=b^q(d^qb-ac^q)(d^{q+1}-a^{q+1}),
        \end{eqnarray}
        which is the right side of Equation (\ref{eqn13}). Hence $\alpha$ satisfies the second equation of (\ref{eqn7}).

        Now we show $\alpha$ is a solution of Equation (\ref{eqn7}), which is equivalent to saying that $d^qD(x)-aN(x)$ is a factor of $N(x)$. Since $d^qD(x)-aN(x)$ is an SCR polynomial, it must also be a factor of $D(x)$, implying that the greatest common divisor is $d^qD(x)-aN(x)$. Hence $\text{deg}(H)=2$.
        \end{proof}
    \begin{lem}\label{lem10}
        A quadratic polynomial $f(x)=x^2+Ax+B\in \mathbb{F}_{q^2}[x]$ is an SCR polynomial if and only if $A^qB=A$ and $B^{q+1}=1$.
    \end{lem}

    \begin{proof}
        If $A^qB=A$ and $B^{q+1}=1$, then $\hat{f}(x)=B^qf(x)$, which implies $f(x)$ is an SCR polynomial. If $f(x)$ is an SCR polynomial, then $\hat{f}(x)=\alpha f(x)$. By comparing coefficients, we obtain $B^q=\alpha$, $A^q=\alpha A$, and $1=\alpha B$, which are equivalent to $A^qB=A$ and $B^{q+1}=1$.
    \end{proof}
\end{case}
Next, we consider when $\text{deg}(H)$ is equal to $1$. If $d^qc\neq ab^q$, then $\left(\frac{d^qc-ab^q}{d^qb-ac^q}\right)^q\frac{d^{q+1}-a^{q+1}}{d^qb-ac^q}\neq \frac{d^qc-ab^q}{d^qb-ac^q}$. Hence from Lemma \ref{lem10}, $d^qD(x)-aN(x)$ is not an SCR polynomial. If $d^{q+1}=a^{q+1}$, then $d^qD(x)-aN(x)=(d^qb-ac^q)x^2+(d^qc-ab^q)x$. The only possible solution is $\frac{ab^q-d^qc}{d^qb-ac^q}$. Since $\left(\frac{ab^q-d^qc}{d^qb-ac^q}\right)^{q+1}=\frac{a^{q+1}b^{q+1}-adb^qc^q-a^qd^qbc+d^{q+1}c^{q+1}}{d^{q+1}b^{q+1}-adb^qc^q-a^qd^qbc+a^{q+1}c^{q+1}}=1$, $\frac{ab^q-d^qc}{d^qb-ac^q}\in \mu_{q+1}$. Assume $\frac{ab^q-d^qc}{d^qb-ac^q}$ is a solution of $N(x)=d^qx^3+c^qx^2+b^qx+a^q$. Then we have a factorization $\left(x+\frac{d^qc-ab^q}{d^qb-ac^q}\right)(x^2+\alpha x+\beta)$. From $d^{q+1}=a^{q+1}$, by comparing coefficients in $N(x)$, we obtain that $\text{deg}(H)=1$ if and only if $-(d^qc-ab^q)^3=(d^qb-ac^q)[(b^{q+1}-c^{q+1})(d^qc-ab^q)+(bd^q-ac^q)(dc^q-a^qb)]$.

If $d^{q+1}\neq a^{q+1}$, then $d^qD(x)-aN(x)=(d^qb-ac^q)x^2+(d^qc-ab^q)x+(d^{q+1}-a^{q+1})$. In this case, $d^qD(x)-aN(x)$ is not an SCR polynomial. If $\text{deg}(H)=1$, then the root of $H(x)$ is in $\mu_{q+1}$. Hence the root must satisfy the following system of equations.
\begin{eqnarray}\label{eqn16}
	\left\{\begin{array}{c}
		x^2+\frac{d^qc-ab^q}{d^qb-ac^q}x+\frac{d^{q+1}-a^{q+1}}{d^qb-ac^q}=0,\\
        \frac{d^{q+1}-a^{q+1}}{db^q-a^qc}x^2+\frac{dc^q-a^qb}{db^q-a^qc}x+1=0.
	\end{array}
	\right.
\end{eqnarray}

According to the system of equations, the root must be a solution of the following equation:

\begin{eqnarray}\label{eqn17}
    \left(\frac{(d^{q+1}-a^{q+1})(d^qc-ab^q)}{(db^q-a^qc)^{q+1}}-\frac{(dc^q-a^qb)(d^qb-ac^q)}{(db^q-a^qc)^{q+1}}\right)x+\frac{(d^{q+1}-a^{q+1})^2-(db^q-a^qc)^{q+1}}{(db^q-a^qc)^{q+1}}=0.
\end{eqnarray}

Let $A_1=(d^{q+1}-a^{q+1})(d^qc-ab^q)-(dc^q-a^qb)(d^qb-ac^q)$ and $A_2=(d^{q+1}-a^{q+1})^2-(db^q-a^qc)^{q+1}$. If $A_1=0$ and $A_2=0$, then $d^qD(x)-aN(x)$ is an SCR polynomial, a contradiction. If $A_1=0$ and $A_2\neq 0$, then Equation (\ref{eqn17}) has no solution. Hence $H(x)=1$. If $A_1\neq 0$, we know the only possible solution is $x=-\frac{A_2}{A_1}$. It is easy to know when $\text{deg}(H)=1$, $\frac{A_2}{A_1}\in \mu_{q+1}$. We have the following lemma to show if $\frac{A_2}{A_1}\in \mu_{q+1}$, then $\text{deg}(H)=1$.

\begin{lem}\label{lem11}
    Assume that the polynomial $F(x)\in \mathbb{F}_{q^2}[x]$ is a linear combination of two arbitrary polynomials $G(x)$ and $\hat{G}(x)$, and that $x^{\text{deg}(G)-\text{deg}(F)}\hat{F}(x)$ is not a scalar multiple of $F(x)$. Then if a root $\alpha\in \mu_{q+1}$ is a common root of both $F(x)$ and $\hat{F}(x)$, then it must be a common root of $G(x)$ and $\hat{G}(x)$.
\end{lem}
\begin{proof}
    Let $F(x)=mG(x)+n\hat{G}(x)$. From the expression, we know that $\text{deg}(F)\leq \text{max} \{\text{deg}(G),\text{deg}(\hat{G})\}\leq \text{deg}(G)$. Since $F(x)=mG(x)+n\hat{G}(x)$, we have $x^{\text{deg}(G)-\text{deg($F$)}}\hat{F}(x)=n^qG(x)+m^q\hat{G}(x)$. If $\alpha$ is not a common root of $G(x)$ and $\hat{G}(x)$, then $\begin{bmatrix}
        G(\alpha)\\
        \hat{G}(\alpha)
    \end{bmatrix}$ is a nonzero solution of Equation (\ref{eqn18}):
    \begin{eqnarray}\label{eqn18}
        \begin{bmatrix}
            m&n\\
            n^q&m^q
            
        \end{bmatrix}\begin{bmatrix}
            G(\alpha)\\
            \hat{G}(\alpha)
        \end{bmatrix}=\begin{bmatrix}
            0\\
            0
        \end{bmatrix}
    .\end{eqnarray}

   Since $x^{\text{deg}(G)-\text{deg}(F)}\hat{F}(x)$ is not a scalar multiple of $F(x)$, the determinant of $\begin{bmatrix}
        m&n\\
        n^q&m^q
    \end{bmatrix}$ is nonzero. This implies that only the zero solution can exist, which leads to a contradiction.
\end{proof}

Using Lemma \ref{lem11}, when $d^{q+1}\neq a^{q+1}$, $A_2/A_1\in \mu_{q+1}$, $-\frac{A_2}{A_1}$ is also a common root of $N(x)$ and $D(x)$, which implies that $\text{deg}(H)=1$.

Finally, we consider $d^qc=ab^q$. In this case, $d^qD(x)-aN(x)=(d^qb-ac^q)x^2+(d^{q+1}-a^{q+1})$. If $d^{q+1}=a^{q+1}$, then $H(x)=1$. If $d^{q+1}\neq a^{q+1}$, it is easy to know if $\frac{d^{q+1}-a^{q+1}}{d^qb-ac^q}\notin \mu_{q+1}$, then $H(x)=1$. If $\frac{d^{q+1}-a^{q+1}}{d^qb-ac^q}\in \mu_{q+1}$ and $\frac{d^{q+1}-a^{q+1}}{d^qb-ac^q}\neq \frac{b^q}{d^q}$, then $H(x)=1$.

From the above discussion, we have the following theorem.

\begin{thm}\label{thm1}
    The $\text{deg}(g)$ can be $0,1,2,3$, characterized as follows:

    (1)$\text{deg}(g)=0$ if and only if one of the following statements holds:

    \hspace{2em}(i) $d=0$, $a=0$, and $b^{q+1}=c^{q+1}$;

    \hspace{2em}(ii) $d\neq 0$, $d^qb=ac^q$, $d^qc=ab^q$, and $d^{q+1}=a^{q+1}$.

    (2) $\text{deg}(g)=1$ if and only if one of the following holds:

    \hspace{2em}(i) $d=0$, $c\neq 0$, $bc^q=ab^q$, and $c^{q+1}=a^{q+1}$;

    \hspace{2em}(ii) $d=0$, $a=0$, and $c^{q+1}\neq b^{q+1}$;

    \hspace{2em}(iii) $d\neq 0$, $d^qb\neq ac^q$, $d^qc\neq ab^q$, and $\left(\frac{d^qc-ab^q}{d^qb-ac^q}\right)^q\frac{d^{q+1}-a^{q+1}}{d^qb-ac^q}=\frac{d^qc-ab^q}{d^qb-ac^q}$;

    \hspace{2em}(iv) $d\neq 0$, $d^qb\neq ac^q$, $d^qc=ab^q$, and $\frac{d^{q+1}-a^{q+1}}{d^qb-ac^q}=\frac{b^q}{d^q}\in \mu_{q+1}$.

    (3) $\text{deg}(g)=2$ if and only if one of the following holds:

    \hspace{2em}(i) $d=0$, $b\neq 0$, $a\neq 0$, $bc^q\neq ab^q$, and $\frac{a^{q+1}-c^{q+1}}{bc^q-ab^q}\in \mu_{q+1}$;

    \hspace{2em}(ii) $d\neq 0$, $d^qb=ac^q$, $d^qc\neq ab^q$, $d^{q+1}\neq a^{q+1}$, and $\frac{d^qc-ab^q}{d^{q+1}-a^{q+1}}\in \mu_{q+1}$;

    \hspace{2em}(iii) $d\neq 0$, $d^qb\neq ac^q$, $d^qc\neq ab^q$, $d^{q+1}=a^{q+1}$, and $-(d^qc-ab^q)^3=(d^qb-ac^q)[(b^{q+1}-c^{q+1})(d^qc-ab^q)+(bd^q-ac^q)(dc^q-a^qb)]$;

    \hspace{2em}(iv) $d\neq 0$, $d^qb\neq ac^q$, $d^qc\neq ab^q$, $d^{q+1}\neq a^{q+1}$, and $\frac{A_2}{A_1}\in \mu_{q+1}$, where $A_1=(d^{q+1}-a^{q+1})(d^qc-ab^q)-(dc^q-a^qb)(d^qb-ac^q)$ and $A_2=(d^{q+1}-a^{q+1})^2-(db^q-a^qc)^{q+1}$.

    (4) In all other cases, $\text{deg}(g)=3$.
\end{thm}

\section{Conditions for $g(x)$ to permute $\mu_{q+1}$}

In this section, we characterize when $g(x)$ permutes $\mu_{q+1}$, assuming that $g(x)$ has degree $0,1,2$, or $3$.

It is easy to see that when $\text{deg}(g)=0$, $g(x)$ is constant on $\mu_{q+1}$, and therefore  cannot permute $\mu_{q+1}$.

\subsection{\text{Deg(g)}=1}
When $\text{deg}(g)=1$, Theorem \ref{thm1} (2) shows there are four cases. We discuss them case by case.

Case (i): $d=0$, $c\neq 0$, $bc^q=ab^q$, and $c^{q+1}=a^{q+1}$. In this case, $g(x)=\frac{c^q}{ax}$. Since $c^{q+1}=a^{q+1}$, we have $\frac{c^q}{a}\in \mu_{q+1}$. Hence $g(x)=\frac{c^q}{ax}$ permutes $\mu_{q+1}$.

Case (ii): $d=0$, $a=0$, and $c^{q+1}\neq b^{q+1}$. In this case, $g(x)=\frac{c^qx+b^q}{bx+c}$. It is easy to see that $g(x)\in \mu_{q+1}$ when $x\in \mu_{q+1}$. Since $c^{q+1}\neq b^{q+1}$, $g(x)$ is an invertible linear transform. Hence $g(x)$ permutes $\mu_{q+1}$.

Case (iii): $d\neq 0$, $d^qb\neq ac^q$, $d^qc\neq ab^q$, and $\left(\frac{d^qc-ab^q}{d^qb-ac^q}\right)^q\frac{d^{q+1}-a^{q+1}}{d^qb-ac^q}=\frac{d^qc-ab^q}{d^qb-ac^q}$. In this case, if $a\neq 0$, then 
$g(x)=\frac{d^qx+\frac{a^q(d^qb-ac^q)}{d^{q+1}-a^{q+1}}}{ax+\frac{d(d^qb-ac^q)}{d^{q+1}-a^{q+1}}}$.
Denote $\alpha=\frac{d(d^qb-ac^q)}{a(d^{q+1}-a^{q+1})}$. From Lemma \ref{lem2} (3), we know $\alpha^{-q}={\frac{a^q(d^qb-ac^q)}{d^q(d^{q+1}-a^{q+1})}}$. Then $g(x)=\frac{d^q(x+\alpha^{-q})}{a(x+\alpha)}$. From Lemma \ref{lem3}, we know $g(x)$ permutes $\mu_{q+1}$ if and only if there exists $\lambda\in \mathbb{F}_{q^2}^{*}$ such that $\frac{a\alpha}{d}=\lambda^{q-1}$ and $(\alpha^{-q})^{q+1}\neq 1$. These two conditions are equivalent to $\frac{a\alpha}{d}\in \mu_{q+1}$ and $\alpha\not\in \mu_{q+1}$. Next, we show that  these two conditions hold.

Since $\alpha=\frac{d(d^qb-ac^q)}{a(d^{q+1}-a^{q+1})}$, we know $\frac{a\alpha}{d}=\frac{d^qb-ac^q}{d^{q+1}-a^{q+1}}$. Since $\left(\frac{d^qc-ab^q}{d^qb-ac^q}\right)^q\frac{d^{q+1}-a^{q+1}}{d^qb-ac^q}=\frac{d^qc-ab^q}{d^qb-ac^q}$, we have 
\begin{eqnarray}\label{eqn19}
    (dc^q-a^qb)(d^{q+1}-a^{q+1})=(d^qc-ab^q)(db^q-a^qc)
\end{eqnarray}
and its Frobenius
\begin{eqnarray}\label{eqn20}
    (d^qc-ab^q)(d^{q+1}-a^{q+1})=(dc^q-a^qb)(d^qb-ac^q).
\end{eqnarray}

By multiplying Equation (\ref{eqn19}) and Equation (\ref{eqn20}) and dividing by $(dc^q-a^qb)(d^qc-ab^q)$, we obtain
\begin{eqnarray}\label{eqn21}
    (d^{q+1}-a^{q+1})^2=(db^q-a^qc)(d^qb-ac^q).
\end{eqnarray}

Since $\frac{a\alpha}{d}=\frac{d^qb-ac^q}{d^{q+1}-a^{q+1}}$, from Equation (\ref{eqn21}), we know 
that $\frac{a\alpha}{d}\in \mu_{q+1}$. Moreover, since in this case $d^{q+1}\neq a^{q+1}$, we conclude that $\alpha\not\in \mu_{q+1}$. Hence,  these two conditions hold, and $g(x)$ permutes $\mu_{q+1}$.

If $a=0$, then $g(x)=\frac{d^qx}{b}$. From $\left(\frac{d^qc-ab^q}{d^qb-ac^q}\right)^q\frac{d^{q+1}-a^{q+1}}{d^qb-ac^q}=\frac{d^qc-ab^q}{d^qb-ac^q}$, we obtain $c^qd=bc$. Since $c\neq 0$, it follows that $\frac{d}{b}\in \mu_{q+1}$. Hence, $\frac{d^q}{b}\in \mu_{q+1}$, and $g(x)$ permutes $\mu_{q+1}$.

Case (iv): $d\neq 0$, $d^qb\neq ac^q$, $d^qc=ab^q$, and $\frac{d^{q+1}-a^{q+1}}{d^qb-ac^q}=\frac{b^q}{d^q}\in \mu_{q+1}$. In this case, if $a\neq 0$, then $g(x)=\frac{d^q(x+\frac{c^q}{d^q})}{a(x+\frac{b}{a})}$. From Lemma \ref{lem3}, we know that $g(x)$ permutes $\mu_{q+1}$ if and only if there exists $\lambda\in \mathbb{F}_{q^2}^{*}$ such that $\frac{c^q}{a^q}=\frac{d^q}{b^q}=\lambda^{q-1}$ and $a^{q+1}\neq b^{q+1}$. From $d^qc=ab^q$ and $\frac{b^q}{d^q}\in \mu_{q+1}$, we obtain $\frac{c^q}{a^q}=\frac{d^q}{b^q}\in \mu_{q+1}$. If $a^{q+1}=b^{q+1}$, then $d^{q+1}=a^{q+1}$, which implies $0\in \mu_{q+1}$, a contradiction. Hence these two conditions hold, and $g(x)$ permutes $\mu_{q+1}$. If $a=0$, then $b\neq 0$ and $c=0$. Hence in this case, $g(x)=\frac{d^qx}{b}$. Since $\frac{d^q}{b}\in \mu_{q+1}$, $g(x)$ permutes $\mu_{q+1}$.

\subsection{\text{Deg(g)}=2}

We have the following classification result for permutation rational functions of degree 2: a degree-2 rational function $f(x)\in \mathbb{F}_q(x)$ permutes $\mathbb{P}^{1}(\mathbb{F}_q)$ if and only if $q$ is even and $f(x)$ is equivalent to $x^2$.

From the classification result for permutation rational functions of degree 2, we only need to consider the case where $q$ is even. Using some degree-1 rational function $\rho$, $\sigma\in \mathbb{F}_{q^2}(x)$ which map $\mu_{q+1}$ to $\mathbb{P}^{1}(\mathbb{F}_q)$, we obtain $\tilde{g}(x)=\rho\circ g\circ \sigma^{-1}\in \mathbb{F}_q(x)$ from Lemma \ref{lem6}. Hence $g(x)$ permutes $\mu_{q+1}$ if and only if $\tilde{g}(x)$ permutes $\mathbb{P}^{1}(\mathbb{F}_q)$. Applying this result to $g(x)$, we obtain $g(x)$ is of the form $\frac{Ax^2+B}{Cx^2+D}$, where $\frac{Ax+B}{Cx+D}$ permutes $\mu_{q+1}$. 

Case (i): $d=0$, $a\neq 0$, $b\neq 0$, $bc^q\neq ab^q$, and $\frac{a^{q+1}-c^{q+1}}{bc^q-ab^q}\in \mu_{q+1}$. If $c=0$, then $g(x)=\frac{a}{b^q}x^2$. From  $\frac{a^{q+1}-c^{q+1}}{bc^q-ab^q}\in \mu_{q+1}$, we have $\frac{a^q}{b^q}\in \mu_{q+1}$, which implies $\frac{a}{b^q}\in \mu_{q+1}$. Hence $g(x)$ permutes $\mu_{q+1}$. If $c\neq 0$, then the degree of the numerator of $g(x)$ is 2, while the degree of the denominator of $g(x)$ is 1. This contradicts the assumption that $g(x)$ is of the form $\frac{Ax^2+B}{Cx^2+D}$.

Case (ii): $d\neq 0$, $d^qb=ac^q$, $d^qc\neq ab^q$, $d^{q+1}\neq a^{q+1}$, and $\frac{d^qc-ab^q}{d^{q+1}-a^{q+1}}\in \mu_{q+1}$. In this case, from Lemma \ref{lem8}, we know $H(x)=x+\frac{d^{q+1}-a^{q+1}}{d^qc-ab^q}$. Hence $g(x)$ permutes $\mu_{q+1}$ if and only if $g(x)$ is of the form $\frac{\left(x+\frac{d^{q+1}-a^{q+1}}{d^qc-ab^q}\right)(Ax^2+B)}{\left(x+\frac{d^{q+1}-a^{q+1}}{d^qc-ab^q}\right)(Cx^2+D)}$. From $g(x)=\frac{N(x)}{D(x)}=\frac{d^qx^3+c^qx^2+b^qx+a^q}{ax^3+bx^2+cx+d}$, by comparing coefficients, we get
\begin{equation} \label{eqn22}
\begin{aligned}
  A &= d^q, \quad &d^q \cdot \frac{d^{q+1} - a^{q+1}}{d^q c - a b^q} = c^q, \quad
  B &= b^q, \quad &b^q \cdot \frac{d^{q+1} - a^{q+1}}{d^q c - a b^q} = a^q, \\
  C &= a,   \quad &a   \cdot \frac{d^{q+1} - a^{q+1}}{d^q c - a b^q} = b,   \quad
  D &= c,   \quad &c   \cdot \frac{d^{q+1} - a^{q+1}}{d^q c - a b^q} = d.
\end{aligned}
\end{equation}

If $a=0$, then from the conditions of Case (ii), we have $b=0$ and $\frac{d}{c}\in \mu_{q+1}$. Substituting these into Equation (\ref{eqn22}), we obtain $A=d^q$, $B=C=0$ and $D=c$ and the other expressions remain valid. Hence $g(x)=\frac{d^q}{c}x^2$, which permutes $\mu_{q+1}$ since $d^{q+1}=c^{q+1}$.

If $a\neq 0$, from Equation (\ref{eqn22}), we have $\frac{c^q}{d^q}=\frac{b}{a}=\frac{d}{c}=\frac{a^q}{b^q}=\frac{d^{q+1}-a^{q+1}}{d^qc-ab^q}\in \mu_{q+1}$. From Lemma \ref{lem3}, we know that $\frac{d^qx+b^q}{ax+c}$ permutes $\mu_{q+1}$ if and only if $a^{q+1}\neq c^{q+1}$. Since $d^{q+1}\neq a^{q+1}$ and $d^{q+1}=c^{q+1}$, it follows that $\frac{d^qx+b^q}{ax+c}$ permutes $\mu_{q+1}$ under these conditions. Hence $g(x)$ permutes $\mu_{q+1}$ if and only if  $\frac{c^q}{d^q}=\frac{b}{a}=\frac{d}{c}=\frac{a^q}{b^q}=\frac{d^{q+1}-a^{q+1}}{d^qc-ab^q}$.

Case (iii): $d\neq 0$, $d^qb\neq ac^q$, $d^qc\neq ab^q$, $d^{q+1}=a^{q+1}$, and $-(d^qc-ab^q)^3=(d^qb-ac^q)[(b^{q+1}-c^{q+1})(d^qc-ab^q)+(bd^q-ac^q)(dc^q-a^qb)]$. In this case, $H(x)=x+\frac{d^qc-ab^q}{d^qb-ac^q}$. Hence $g(x)$ permutes $\mu_{q+1}$ if and only of $g(x)$ is of the form $\frac{\left(x+\frac{d^qc-ab^q}{d^qb-ac^q}\right)(Ax^2+B)}{\left(x+\frac{d^qc-ab^q}{d^qb-ac^q}\right)(Cx^2+D)}$. From $g(x)=\frac{N(x)}{D(x)}=\frac{d^qx^3+c^qx^2+b^qx+a^q}{ax^3+bx^2+cx+d}$, by comparing coefficients, we get
\begin{equation} \label{eqn23}
\begin{aligned}
  A &= d^q, \quad &d^q \cdot \frac{d^qc-ab^q}{d^qb-ac^q} = c^q, \quad
  B &= b^q, \quad &b^q \cdot \frac{d^qc-ab^q}{d^qb-ac^q} = a^q, \\
  C &= a,   \quad &a   \cdot \frac{d^qc-ab^q}{d^qb-ac^q} = b,   \quad
  D &= c,   \quad &c   \cdot \frac{d^qc-ab^q}{d^qb-ac^q} = d.
\end{aligned}
\end{equation}

 From Equation (\ref{eqn23}), we have $\frac{c^q}{d^q}=\frac{b}{a}=\frac{d}{c}=\frac{a^q}{b^q}\in \mu_{q+1}$. From Lemma \ref{lem3}, we know that $\frac{d^qx+b^q}{ax+c}$ permutes $\mu_{q+1}$ if and only if $a^{q+1}\neq c^{q+1}$. Since $c^{q+1}=d^{q+1}=a^{q+1}$, it follows that $\frac{d^qx+b^q}{ax+c}$ does not permute $\mu_{q+1}$. Hence $g(x)$ does not permute $\mu_{q+1}$.

 Case (iv): $d\neq 0$, $d^qb\neq ac^q$, $d^qc\neq ab^q$, $d^{q+1}\neq a^{q+1}$, and $\frac{A_2}{A_1}\in \mu_{q+1}$, where $A_1=(d^{q+1}-a^{q+1})(d^qc-ab^q)-(dc^q-a^qb)(d^qb-ac^q)$ and $A_2=(d^{q+1}-a^{q+1})^2-(db^q-a^qc)^{q+1}$. Similar to Case (ii), when $a\neq 0$, we know that $g(x)$ permutes $\mu_{q+1}$ if and only if $\frac{c^q}{d^q}=\frac{b}{a}=\frac{d}{c}=\frac{a^q}{b^q}=\frac{A_2}{A_1}\in \mu_{q+1}$ . If $a\neq 0$, since $d^qb\neq ac^q$, we cannot get $\frac{c^q}{d^q}=\frac{b}{a}$, which shows that $g(x)$ does not permute $\mu_{q+1}$. When $a=0$, if $g(x)$ permutes $\mu_{q+1}$, then $b=0$, which contradicts $d^qb\neq ac^q$.

\subsection{\text{Deg(g)}=3}

From now on, we assume that the rational function $g(x)\in \mathbb{F}_q(x)$ has degree $3$. According to Lemma \ref{lem7}, which  classifies separable permutation rational functions of degree $3$, we know that when the characteristic of the finite field $\mathbb{F}_q$ is not $3$, any separable permutation rational function of degree $3$ is equivalent over $\mathbb{F}_{q^2}$ to $x^3$. In contrast, when the characteristic is $3$, any separable permutation rational function of degree $3$ is equivalent over $\mathbb{F}_q$ to $x^3-\alpha x$, where $\alpha$ is a nonsquare element in $\mathbb{F}_q$. Therefore, we will now discuss the two cases separately: when the characteristic does not divide 3 and when it does.

\textbf{Case 1} $3\nmid q$.

In this case, any rational function $g(x)$ of degree $3$ is separable since $g(x)\not\in \mathbb{F}_q(x^p)$, where $p$ is the characteristic of $\mathbb{F}_q$. From Lemma \ref{lem6}, we know if $g(x)=\frac{N(x)}{D(x)}\in \mathbb{F}_{q^2}(x)$ permutes $\mu_{q+1}$, then for any degree-$1$ rational functions $\rho,\sigma\in \mathbb{F}_{q^2}(x)$ that map $\mu_{q+1}$ to $\mathbb{P}^{1}(\mathbb{F}_q)$, the function $h(x):=\rho\circ g\circ \sigma^{-1}$ permutes $\mathbb{P}^{1}(\mathbb{F}_q)$, and $h(x)\in \mathbb{F}_q(x)$. According to Lemma \ref{lem7}, we know $h(x)$ is equivalent over $\mathbb{F}_{q^2}$ to $x^3$. Note that $x^3$ has two branch points in $\mathbb{P}^{1}(\mathbb{F}_q)$: $0$ and $\infty$, each of which has a unique preimage $0$ and $\infty$ respectively. Since the number of branch points and the distribution of multiset $E_g(\beta)$, as $\beta$ ranges over  $\mathbb{P}^{1}(\bar{\mathbb{F}}_q)$, are invariant under the composition with degree-1 rational functions, both $h(x)$ and $g(x)$ each have exactly two branch points in $\mathbb{P}^{1}(\bar{\mathbb{F}}_q)$, and each branch point has a unique preimage. Therefore, we obtain a necessary condition for $g(x)$ to induce a permutation of $\mu_{q+1}$ when $3\nmid q$.

\begin{lem}\label{lem12}
When $3\nmid q$, if $g(x)$ is a degree-$3$ permutation of $\mu_{q+1}$, then $g(x)$ has two branch points in $\mathbb{P}^{1}(\bar{\mathbb{F}}_q)$, each of which has a unique preimage.
\end{lem}

Now we want to use the above lemma to deduce some necessary conditions on the coefficients of $g(x)$.

First, assume that $g(x)$ has a branch point $\infty$, which has a unique preimage. Then in $\mathbb{P}^{1}(\bar{\mathbb{F}}_q)$, $D(x)$ can be factorized as $a(x-\alpha)^3$, where $\alpha$ is the unique preimage. Since $N(x)=\hat{D}(x)$, it follows from Lemma \ref{lem2} that $N(X)$ can be factorized as $d^q(x-\alpha^{-q})^3$. Hence, $g(x)=\frac{d^q(x-\alpha^{-q})^3}{a(x-\alpha)^3}$.  If $\alpha=\infty$, then $\alpha^{-q}=0$, which implies $g(x)=\frac{d^q}{a}x^3$, contradicting the fact that $g(x)$ is of the form $\frac{N(x)}{D(x)}=\frac{d^qx^3+c^qx^2+b^qx+a^q}{ax^3+bx^2+cx+d}$. Hence, $\alpha\in \bar{\mathbb{F}}_q$.
Since coprimality and the derivative of a polynomial are independent of field extensions, $D(x)$ has no repeated factors in $\mathbb{F}_{q^2}[x]$ if and only if it has no repeated factors in $\bar{\mathbb{F}}_q[x]$. Hence, $D(x)$ has repeated factors in $\mathbb{F}_{q^2}[x]$, which implies $\alpha\in \mathbb{F}_{q^2}$. Since $\text{deg}(g)=3$, we have $\alpha\not\in \mu_{q+1}$. Therefore, $g(x)$ can be written as $\frac{d^q}{a}x\circ x^3\circ \frac{x-\alpha^{-q}}{x-\alpha}$. Hence, $g(x)$ permutes $\mu_{q+1}$ if and only if $\frac{x-\alpha^{-q}}{x-\alpha}$, $x^3$, and $\frac{d^q}{a}$ permute $\mu_{q+1}$.  By Lemma \ref{lem3}, the rational function $\frac{x-\alpha^{-q}}{x-\alpha}$ permutes $\mu_{q+1}$ if and only if there exists $\lambda\in \mathbb{F}_{q^2}^{*}$ such that $-\lambda^q\alpha^q=\lambda$. This condition implies that $\alpha\in \mu_{q+1}$, yielding a contradiction.

Next, assume that $\beta\in \bar{\mathbb{F}}_q$ is a branch point of $g(x)$ with a unique preimage $\alpha$. This is equivalent to saying that $\alpha$ is a triple root of $F(x)=N(x)-\beta D(x)$. Since $\alpha$ is a triple root of $F(x)=N(x)-\beta D(x)$, it follows that $F(\alpha)=0$, $F'(\alpha)=0$ and $F''(\alpha)=0$. We now write out the explicit expressions for $F(x)$, $F'(x)$, and $F''(x)$:
\begin{eqnarray}\label{eqn24}
    F(x)=N(x)-\beta D(x)=(d^q-\beta a)x^3+(c^q-\beta b)x^2+(b^q-\beta c)x+(a^q-\beta d),
\end{eqnarray}
\begin{eqnarray}\label{eqn25}
    F'(x)=N'(x)-\beta D'(x)=3(d^q-\beta a)x^2+2(c^q-\beta b)x+(b^q-\beta c),
\end{eqnarray}
\begin{eqnarray}\label{eqn26}
    F''(x)=N''(x)-\beta D''(x)=6(d^q-\beta a)x+2(c^q-\beta b).
\end{eqnarray}

When the characteristic of finite field $\mathbb{F}_q$ is not $2$, since $F''(\alpha)=0$, it follows from Equation (\ref{eqn26}) that 
\begin{eqnarray}\label{eqn27}
    \alpha=\frac{-(c^q-\beta b)}{3(d^q-\beta a)}.
\end{eqnarray}

Given that $F'(\alpha)=0$, substituting Equation (\ref{eqn27}) into the expression $F'(\alpha)=0$, we have
\begin{eqnarray}\label{eqn28}
    (3ac-b^2)\beta^2+(2bc^q-3d^qc-3ab^q)\beta+3d^qb^q-c^{2q}=0.
\end{eqnarray}

From Equation (\ref{eqn28}), we conclude that any branch point $\beta$ must satisfy the following equation:

\begin{eqnarray}\label{eqn29}
    (3ac-b^2)x^2+(2bc^q-3d^qc-3ab^q)x+3d^qb^q-c^{2q}=0.
\end{eqnarray}

When the characteristic of finite field $\mathbb{F}_q$ is $2$, $F''(x)\equiv 0$. From Equation (\ref{eqn25}), we have
\begin{eqnarray}\label{eqn30}
\alpha^2=\frac{b^q-\beta c}{d^q-\beta a}.
\end{eqnarray}

Plugging Equation (\ref{eqn30}) into $F(\alpha)=0$, we have
\begin{eqnarray}\label{eqn31}
    (ad+bc)\beta^2+(a^{q+1}+b^{q+1}+c^{q+1}+d^{q+1})\beta+(a^qd^q+b^qc^q)=0.
\end{eqnarray}

From Equation (\ref{eqn30}), we conclude that any branch point $\beta$ must satisfy the following equation:

\begin{eqnarray}\label{eqn32}
     (ad+bc)x^2+(a^{q+1}+b^{q+1}+c^{q+1}+d^{q+1})x+(a^qd^q+b^qc^q)=0.
\end{eqnarray}

We have determined the algebraic equation satisfied by the branch points. The following lemma can help us determine the algebraic equation satisfied by the ramification points.

\begin{lem}\cite{ding2025class}\label{lem13}
Let $K$ be a field of characteristic $p\geq 0$, pick a nonconstant $g(x)\in K(x)$, and pick coprime $N,D\in K[x]$ with $g(x)=\frac{N(x)}{D(x)}$. Assume that $R(x):=N'(x)D(x)-N(x)D'(x)$ is nonzero. Pick $\alpha\in K$ for which $D(\alpha)\neq 0$ and write $\beta:=g(\alpha)$. Then the following hold:

(1) $\alpha$ is a root of $R(x)$ if and only if $e_g(\alpha)>1$.

(2) If $e_g(\alpha)>1$, then $e_R(\alpha)\geq e_g(\alpha)-1$ with equality holding if and only if $p\nmid e_g(\alpha)$.
\end{lem}

Since in our case $\alpha$ is the unique preimage of $\beta$, we have $e_g(\alpha)=3$. From Lemma (\ref{lem13}) (2), we know that $e_R(\alpha)=2$, and from Lemma (\ref{lem13}) (1), we know that only a ramification point of $g(x)$ can be a root of $R(x)$.

Hence we have the following factorization:
\begin{eqnarray}\label{eqn33}
    R(x)=C(x^2+Ax+B)^2,
\end{eqnarray}

where $R(x):=N'(x)D(x)-N(x)D'(x)$,  which equals 
\begin{dmath}\label{eqn34}
    (bd^q-ac^q)x^4+(2cd^q-2ab^q)x^3+(c^{q+1}-b^{q+1}+3d^{q+1}-3a^{q+1})x^2+(2c^qd-2a^qb)x+(b^qd-a^qc).
\end{dmath}

From Equation (\ref{eqn33}), we know that 
\begin{eqnarray}\label{eqn35}
    A=\frac{cd^q-ab^q}{bd^q-ac^q},
\end{eqnarray}
and 
\begin{eqnarray}\label{eqn36}
    B=\frac{c^qd-a^qb}{cd^q-ab^q},
\end{eqnarray}
provided that $bd^q\neq ac^q$ and $cd^q\neq ab^q$.

If $bd^q\neq ac^q$ and $cd^q=ab^q$, then according to Equation (\ref{eqn33}), we have $A=0$ and 
\begin{eqnarray}\label{eqn37}
    B=\frac{c^{q+1}-b^{q+1}+3d^{q+1}-3a^{q+1}}{2(bd^q-ac^q)}.
\end{eqnarray}

From the factorization of $R(x)$, we can derive some necessary conditions for $g(x)$ to permute $\mu_{q+1}$.

\begin{pro}\label{pro1}
    If $g(x)=\frac{N(x)}{D(x)}=\frac{d^qx^3+c^qx^2+b^qx+a^q}{ax^3+bx^2+cx+d}$ permutes $\mu_{q+1}$ and $\text{deg}(g)=3$, then $bd^q\neq ac^q$, and one of the following statements holds:

    (i) $cd^q\neq ab^q$, $A^2+2B=\frac{c^{q+1}-b^{q+1}+3d^{q+1}-3a^{q+1}}{bd^q-ac^q}$ and $B^2=\frac{b^qd-a^qc}{bd^q-ac^q}$;

    (ii) $cd^q=ab^q$ and $B^2=\frac{b^qd-a^qc}{bd^q-ac^q}$, $A=0$,

    where $A$ and $B$ are defined as above.
\end{pro}

\begin{proof}
    By expanding the right side of Equation (\ref{eqn33}) and comparing the coefficients, we obtain the result.
\end{proof}

Proposition \ref{pro1} only tells us that when $g(x)$ has two branch points, each with a unique preimage, this is not sufficient to ensure that $g(x)$ permutes $\mu_{q+1}$. The following lemma tells us how much further we need to go.

\begin{lem}\label{lem14}\cite{ding2023determination}
    Let $q$ be a power of a prime $p$, and assume that $g(x):=x^rA^{(q)}(1/x)/A(x)\in \mathbb{F}_{q^2}(x)$ has degree $n\geq 1$, where $r\in \mathbb{Z}$ and $A(x)\in \mathbb{F}_{q^2}[x]$ has no roots in $\mu_{q+1}$. Suppose there exist distinct $\beta_1$, $\beta_2\in \mathbb{P}^{1}(\bar{\mathbb{F}}_q)$ such that $\beta_i$ has a unique $g$-preimage $\alpha_i\in \mathbb{P}^{1}(\bar{\mathbb{F}}_q)$ for each $i$. Then the following are equivalent:

    (1) $g(x)$ permutes $\mu_{q+1}$.

    (2) At least one of the following holds:

    $\text{gcd}(n,q-1)=1$ and at least one $\alpha_i$ is in $\mu_{q+1}$;

    $\text{gcd}(n,q+1)=1$ and at least one $\alpha_i$ is not in $\mu_{q+1}$.

    (3) At least one of the following holds:

    $\text{gcd}(n,q-1)=1$ and at least one $\beta_i$ is in $\mu_{q+1}$;

    $\text{gcd}(n,q+1)=1$ and at least one $\beta_i$ is not in $\mu_{q+1}$.

    (4) At least one of the following holds:

    $\text{gcd}(n,q-1)=1$ and $g(x)=\rho^{-1}\circ x^n\circ \sigma$ for some degree-$1$ $\rho$, $\sigma\in \mathbb{F}_{q^2}(x)$ which map $\mu_{q+1}$ to $\mathbb{P}^{1}(\mathbb{F}_q)$;

    $\text{gcd}(n,q+1)=1$ and $g(x)=\rho^{-1}\circ x^n\circ \sigma$ for some degree-$1$ $\rho$, $\sigma\in \mathbb{F}_{q^2}(x)$ which permute $\mu_{q+1}$.

    Moreover, if $g(x)=g_1(x^{p^l})$ where $l\geq 0$ and $g_1(x)\in \mathbb{F}_{q^2}(x)\setminus \mathbb{F}_{q^2}(x^p)$ has degree at least $2$, then in (4) we may require in addition that $\sigma(\alpha_1)=\infty=\rho(\beta_1)$ and $\sigma(\alpha_2)=0=\rho(\beta_2)$.
\end{lem}

According to Lemma \ref{lem14}, we aim to determine whether any of the $\beta_i$ lie in $\mu_{q+1}$, where the $\beta_i$ satisfy Equation (\ref{eqn29}) when the characteristic of $\mathbb{F}_q$ is not $2$, and Equation (\ref{eqn32}) when the characteristic of $\mathbb{F}_q$ is $2$.

We have the following lemma to determine whether a quadratic equation has a solution in $\mu_{q+1}$.

\begin{lem}\label{lem15}
    Let $L(x)=x^2+Ax+B\in \mathbb{F}_{q^2}[x]$, where $B\neq 0$. Then $L(x)$ has a solution in $\mu_{q+1}$ if and only if the following equation system has solutions:
    \begin{eqnarray}\notag
           \begin{cases}
        L(x)=x^{2}+Ax+B\\
        M(x)=Bx^q+x+A.
    \end{cases}
    \end{eqnarray}
 
\end{lem}
\begin{proof}
    If $L(x)$ has a solution $\alpha$ in $\mu_{q+1}$, then $\alpha^2+A\alpha+B=0$. Multiply both sides by $\alpha^{q}$, we have $\alpha+A+B\alpha^q=0$, which implies that $\alpha$ is also a solution of $M(x)$.

    Conversely, suppose $L(x)$ and $M(x)$ have a common solution $\alpha$. We will show $\alpha\in \mu_{q+1}$. Multiply both sides of $L(x)$ by $\alpha^q$ and subtract $M(x)$, we have $(A+\alpha)(\alpha^{q+1}-1)=0$. If $A+\alpha=0$, then $A=\alpha=0$ since $B\neq 0$. Substituting $A=\alpha=0$ into $L(x)$ gives $B=0$, which is a contradiction. Hence $A+\alpha\neq 0$, and therefore $\alpha\in \mu_{q+1}$.
\end{proof}

From Lemma \ref{lem15}, we can determine whether any of the $\beta_i$ lie in $\mu_{q+1}$. Hence we obtain a necessary and sufficient condition for $g(x)$ permutes $\mu_{q+1}$ when the characteristic of $\mathbb{F}_q$ is not $3$.

\begin{thm}\label{thm2}
    Let $g(x)=\frac{N(x)}{D(x)}=\frac{d^qx^3+c^qx^2+b^qx+a^q}{ax^3+bx^2+cx+d}\in \mathbb{F}_{q^2}(x)$. Assume that $\text{deg}(g)=3$. When the characteristic of $\mathbb{F}_q$ is $2$, the function $g(x)$ permutes $\mu_{q+1}$ if and only if Proposition \ref{pro1} holds and either
    
    (1) $\text{gcd}(3,q-1)=1$ and Equation (\ref{eqn32}) has a root in $\mu_{q+1}$, or 
    
    (2) $\text{gcd}(3,q+1)=1$ and Equation (\ref{eqn32}) has no roots in $\mu_{q+1}$. 
    
    When the characteristic of $\mathbb{F}_q$ is neither $2$ or $3$, the function $g(x)$ permutes $\mu_{q+1}$ if and only if Proposition \ref{pro1} holds and either 
    
    (1) $\text{gcd}(3,q-1)=1$ and Equation (\ref{eqn29}) has a root in $\mu_{q+1}$, or 
    
    (2) $\text{gcd}(3,q+1)=1$ and Equation (\ref{eqn29}) has no roots in $\mu_{q+1}$.
\end{thm}

\textbf{Case 2} $3\mid q$.

    In this case, if the rational function $g(x)$ is not separable, then $g(x)\in\mathbb{F}_{q^2}(x^3)$, which means $g(x)=f(x^3)$, where $f(x)$ is a degree-1 rational function of the form $\frac{d^qx+a^q}{ax+d}$. From Lemma \ref{lem3}, we know that $f(x)$ permutes $\mu_{q+1}$ if and only if $a^{q+1}\neq d^{q+1}$. Since $x^3$ is a permutation of $\mu_{q+1}$, it follows that $g(x)$ permutes $\mu_{q+1}$ if and only if $f(x)$ does.

    Hence, we now assume that the rational function $g(x)$ is separable. From Lemma \ref{lem6}, we know if $g(x)=\frac{N(x)}{D(x)}\in \mathbb{F}_{q^2}(x)$ permutes $\mu_{q+1}$, then for any degree-$1$ rational functions $\rho,\sigma\in \mathbb{F}_{q^2}(x)$ that map $\mu_{q+1}$ to $\mathbb{P}^{1}(\mathbb{F}_q)$, the function $h(x)=\rho\circ g\circ \sigma^{-1}$ permutes $\mathbb{P}^{1}(\mathbb{F}_q)$, and $h(x)\in \mathbb{F}_q(x)$. According to Lemma \ref{lem7}, we know $h(x)$ is linearly equivalent to $x^3-\alpha x$ for some nonsquare $\alpha\in \mathbb{F}_q^{*}$. Note that $x^3-\alpha x$ has one branch point in $\mathbb{P}^{1}(\mathbb{F}_q)$: $\infty$, which has a unique preimage $\infty$. Since the number of branch points and the distribution of multiset $E_g(\beta)$, as $\beta$ ranges over $\mathbb{P}^{1}(\bar{\mathbb{F}}_q)$, are invariant under the composition with degree-1 rational functions, both $h(x)$ and $g(x)$ each have one branch point in $\mathbb{P}^{1}(\bar{\mathbb{F}}_q)$, and this branch point has a unique preimage. Hence, we obtain a necessary condition for $g(x)$ to induce a permutation of $\mu_{q+1}$ when $3|q$.

    \begin{lem}\label{lem16}
        When $3|q$, if $g(x)$ is a degree-$3$ permutation of $\mu_{q+1}$, then $g(x)$ has a branch point in $\mathbb{P}^{1}(\bar{\mathbb{F}}_q)$, which has a unique preimage.
    \end{lem}

    Now we want to use the above lemma to deduce some necessary conditions on the coefficients of $g(x)$.

    First, assume that $g(x)$ has a branch point $\infty$, which has a unique image. Then in $\mathbb{P}^{1}(\bar{\mathbb{F}}_q)$, $D(x)$ can be factorized as $a(x-\alpha)^3$, where $\alpha$ is the unique preimage. Since $N(x)=\hat{D}(x)$, it follows from Lemma \ref{lem2} that $N(x)$ can be factorized as $d^q(x-\alpha^{-q})$. This means that $g(x)$ has another branch point, which is a contradiction.

    Now, assume that $g(x)$ has a branch point $\beta$ and its unique $g$-preimage $\alpha$. From Lemma $\ref{lem5}$, we know $g(\mu_{q+1})\subseteq \mu_{q+1}$ since $\text{deg}(g)=3$ and $D(x)$ has no roots in $\mu_{q+1}$. Since $g(x)=x^{-1}\circ g^{(q)}(x)\circ x^{-1}$, the unique $g$-preimage of $\beta^{-q}$ in $\mathbb{P}^{1}(\bar{\mathbb{F}}_q)$ is $\alpha^{-q}$. As $g(x)$ has only one branch point, we must have $\alpha=\alpha^{-q}$ and $\beta=\beta^{-q}$, which implies $\alpha\in \mu_{q+1}$ and $\beta\in \mu_{q+1}$.

    Now we determine the exact values of $\alpha$ and $\beta$.
    As in the previous case, since $\alpha$ is a triple root of $F(x)=N(x)-\beta D(x)$, it follows that $F(\alpha)=0$, $F'(\alpha)=0$ and $F''(\alpha)=0$. 
    \begin{eqnarray}\label{eqn38}
        F(x)=N(x)-\beta D(x)=(d^q-\beta a)x^3+(c^q-\beta b)x^2+(b^q-\beta c)x+(a^q-\beta d),
    \end{eqnarray}
    \begin{eqnarray}\label{eqn39}
        F'(x)=N'(x)-\beta D'(x)=2(c^q-\beta b)x+(b^q-\beta c),
    \end{eqnarray}
    \begin{eqnarray}\label{eqn40}
        F''(x)=N''(x)-\beta D''(x)=2(c^q-\beta b).
    \end{eqnarray}

    From Equation (\ref{eqn40}), we know that $c^q=\beta b$; in other words, $\beta=\frac{c^q}{b}$. Substituting $c^q=\beta b$ into $F'(\alpha)=0$, we obtain $b^q-\beta c=0$.

    Substituting $b^q-\beta c=0$ and $c^q-\beta b=0$ into $F(\alpha)=0$, we obtain $\alpha=\left(\frac{\beta d-a^q}{d^q-\beta a}\right)^{\frac{1}{3}}$. We can also obtain $\alpha$ using the factorization of $R(x)=N'(x)D(x)-N(x)D'(x)$ from Lemma \ref{lem13}. In this case, 
    \begin{eqnarray}\label{eqn41}
        R(x)=A(x-\alpha)^4.
    \end{eqnarray}

    From Equation (\ref{eqn41}), we know that
    \begin{eqnarray}\label{eqn42}
        \alpha=\frac{cd^q-ab^q}{bd^q-ac^q}.
    \end{eqnarray}

    From the factorization of $R(x)$, we can derive some necessary conditions for $g(x)$ to permute $\mu_{q+1}$.

    \begin{pro}\label{pro2}
       When the characteristic of $\mathbb{F}_q$ is $3$, if $g(x)=\frac{N(x)}{D(x)}=\frac{d^qx^3+c^qx^2+b^qx+a^q}{ax^3+bx^2+cx+d}$ permutes $\mu_{q+1}$ and $\text{deg}(g)=3$, then $bd^q\neq ac^q$ and $\left(\frac{cd^q-ab^q}{bd^q-ac^q}\right)^3=\frac{c^qd-a^qb}{bd^q-ac^q}$.
    \end{pro}
    Use bijection from $\mu_{q+1}$ to $\mathbb{P}^{1}(\mathbb{F}_q)$ given in Lemma \ref{lem4}, we can transform the rational function $g(x)$ into a degree-$3$ polynomial $h(x)$ on $\mathbb{P}^{1}(\mathbb{F}_q)$ via the composition $h=\rho\circ g\circ \sigma^{-1}$, where $\sigma(\alpha)=\rho(\beta)=\infty$.

    Note that $\mu_{q+1}\cap \mathbb{F}_{q}=\{\pm1\}$. We divide the discussion into four cases according to whether $\alpha$ or $\beta$ lie in $\mathbb{F}_q$ or not: (1): $\alpha\not\in \mathbb{F}_q, \beta\not\in \mathbb{F}_q$; (2): $\alpha\not\in \mathbb{F}_q, \beta\in \mathbb{F}_q$;
    (3): $\alpha\in \mathbb{F}_q, \beta\not\in \mathbb{F}_q$;
    (4): $\alpha\in \mathbb{F}_q, \beta\in \mathbb{F}_q$.

    (1): $\alpha\not\in \mathbb{F}_q, \beta\not\in \mathbb{F}_q$. In this case, we can choose $\sigma=\frac{\alpha x-1}{x-\alpha}$ and $\rho=\frac{\beta x-1}{x-\beta}$. It is easy to know $\sigma^{-1}=\sigma$. Now we compute $h(x)=\rho\circ g\circ \sigma^{-1}$. First, we compute $g\circ \sigma^{-1}(x)$, that is
    \begin{eqnarray}\label{eqn43}
        \begin{aligned}
            g\circ \sigma^{-1}(x)&=\frac{d^q\left(\frac{\alpha x-1}{x-\alpha}\right)^3+c^q\left(\frac{\alpha x-1}{x-\alpha}\right)^2+b^q\left(\frac{\alpha x-1}{x-\alpha}\right)+a^q}{a\left(\frac{\alpha x-1}{x-\alpha}\right)^3+b\left(\frac{\alpha x-1}{x-\alpha}\right)^2+c\left(\frac{\alpha x-1}{x-\alpha}\right)+d}\\
            &=\frac{d^q(\alpha^3x^3-1)+c^q(\alpha x-1)^2(x-\alpha)+b^q(\alpha x-1)(x-\alpha)^2+a^q(x^3-\alpha^3)}{a(\alpha^3 x^3-1)+b(\alpha x-1)^2(x-\alpha)+c(\alpha x-1)(x-\alpha)^2+d(x^3-\alpha^3)}\\
            &=\frac{A_1x^3+A_2x^2+A_3x+A_4}{B_1x^3+B_2x^2+B_3x+B_4},
        \end{aligned}
    \end{eqnarray}
    where \begin{eqnarray}\label{eqn44}
        A_1=(d^q\alpha^3+c^q\alpha^2+b^q\alpha+a^q),
    \end{eqnarray}
    \begin{eqnarray}\label{eqn45}
        A_2=-(c^q(\alpha^3+2\alpha)+b^q(2\alpha^2+1)),
    \end{eqnarray}
    \begin{eqnarray}\label{eqn46}
        A_3=c^q(2\alpha^2+1)+b^q(2\alpha+\alpha^3),
    \end{eqnarray}
    \begin{eqnarray}\label{eqn47}
        A_4=-(d^q+c^q\alpha+b^q\alpha^2+a^q\alpha^3),
    \end{eqnarray}
    \begin{eqnarray}\label{eqn48}
        B_1=a\alpha^3+b\alpha^2+c\alpha+d,
    \end{eqnarray}
    \begin{eqnarray}\label{eqn49}
        B_2=-(b(\alpha^3+2\alpha)+c(2\alpha^2+1)),
    \end{eqnarray}
    \begin{eqnarray}\label{eqn50}
        B_3=b(2\alpha^2+1)+c(2\alpha+\alpha^3),
    \end{eqnarray}
    \begin{eqnarray}\label{eqn51}
        B_4=-(a+b\alpha+c\alpha^2+d\alpha^3).
    \end{eqnarray}

    This gives
    \begin{eqnarray}\label{eqn52}
    \begin{aligned}
         h(x)=\rho\circ g\circ \sigma^{-1}&=\frac{\beta\left(\frac{A_1x^3+A_2x^2+A_3x+A_4}{B_1x^3+B_2x^2+B_3x+B_4}\right)-1}{\left(\frac{A_1x^3+A_2x^2+A_3x+A_4}{B_1x^3+B_2x^2+B_3x+B_4}\right)-\beta}\\
         &=\frac{\beta(A_1x^3+A_2x^2+A_3x+A_4)-(B_1x^3+B_2x^2+B_3x+B_4)}{A_1x^3+A_2x^2+A_3x+A_4-\beta(B_1x^3+B_2x^2+B_3x+B_4)}\\
         &=\frac{(\beta A_1-B_1)x^3+(\beta A_2-B_2)x^2+(\beta A_3-B_3)x+(\beta A_4-B_4)}{(A_1-\beta B_1)x^3+(A_2-\beta B_2)x^2+(A_3-\beta B_3)x+(A_4-\beta B_4)}.
    \end{aligned}
    \end{eqnarray}

    Since $h(x)$ is a degree-$3$ polynomial, we have $A_1=\beta B_1$, $A_2=\beta B_2$, $A_3=\beta B_3$ and $A_4\neq \beta B_1$. Since $g(\alpha)=\beta$, we have $A_1=\beta B_1$. Since $b^q=\beta c$ and $c^q=\beta b$, we have $A_2=\beta B_2$ and $A_3=\beta B_3$. Since $A_4\neq \beta B_4$, we have
    \begin{eqnarray}\label{eqn53}
        (d^q-\beta a)^2-(a^q-\beta d)^2\neq 0.
    \end{eqnarray}

    Since $g(x)$ permutes $\mu_{q+1}$ if and only if $h(x)$ permutes $\mathbb{P}^{1}(\mathbb{F}_q)$. According to Lemma \ref{lem7}, we know $h(x)$ is linearly equivalent to $x^3-\gamma x$ for some nonsquare $\gamma\in \mathbb{F}_q^{*}$. We have $\beta A_1\neq B_1$, $\beta A_2=B_2$, $\beta A_3\neq B_3$. Since $A_1=\beta B_1$, if $\beta A_1=B_1$, we have $\beta^2=1$, which is a contradiction. Since $\beta A_2=B_2$, we have 
    \begin{eqnarray}\label{eqn54}
        (\alpha^2-1)[(\beta c^q-b)\alpha-(\beta b^q-c)]=0.
    \end{eqnarray}


   It is not hard to show that 
    \begin{eqnarray}\label{eqn55}
        \alpha=\frac{cd^q-ab^q}{bd^q-ac^q}=\frac{b^q}{c^q}=\frac{c-\beta b^q}{b-\beta c^q}.
    \end{eqnarray}

    Hence this condition is automatically satisfied when $\alpha=\frac{b^q}{c^q}$.

 Since $\alpha=\frac{c-\beta b^q}{b-\beta c^q}$ and $\alpha^2\neq 1$, we have $\alpha\neq \frac{\beta c^q-b}{\beta b^q-c}$, which is equivalent to $\beta A_3\neq B_3$. Since $\gamma$ is a nonsquare of $\mathbb{F}_q$, we have $\frac{B_3-\beta A_3}{\beta A_1-B_1}$ is a nonsquare of $\mathbb{F}_q^{*}$.

 (2): $\alpha\not\in \mathbb{F}_q$, $\beta\in \mathbb{F}_q$. In this case, we can choose $\sigma=\frac{\alpha x-1}{x-\alpha}$ and $\rho=\frac{\alpha x-\beta \alpha^q}{x-\beta}$. The composition $g\circ \sigma^{-1}(x)$ is the same as in the previous case, namely, $\frac{A_1x^3+A_2x^2+A_3x+A_4}{B_1x^3+B_2x^2+B_3x+B_4}$.

 Now $h(x)$ is
 \begin{eqnarray}\label{eqn56}
 \begin{aligned}
      h(x)&=\frac{\alpha\left(\frac{A_1x^3+A_2x^2+A_3x+A_4}{B_1x^3+B_2x^2+B_3x+B_4}\right)-\beta \alpha^q}{\left(\frac{A_1x^3+A_2x^2+A_3x+A_4}{B_1x^3+B_2x^2+B_3x+B_4}\right)-\beta}\\
      &=\frac{\alpha(A_1x^3+A_2x^2+A_3x+A_4)-\beta \alpha^q(B_1x^3+B_2x^2+B_3x+B_4)}{A_1x^3+A_2x^2+A_3x+A_4-\beta(B_1x^3+B_2x^2+B_3x+B_4)}\\
      &=\frac{(\alpha A_1-\beta \alpha^q B_1)x^3+(\alpha A_2-\beta \alpha^q B_2)x^2+(\alpha A_3-\beta \alpha^q B_3)x+(\alpha A_4-\beta \alpha^q B_4)}{(A_1-\beta B_1)x^3+(A_2-\beta B_2)x^2+(A_3-\beta B_3)x+(A_4-\beta B_4)}.
 \end{aligned} 
 \end{eqnarray}
  By the same discussion as in the above, we have $A_1=\beta B_1$, $A_2=\beta B_2$, $A_3=\beta B_3$, and $A_4\neq \beta B_4$. From Equation (\ref{eqn53}), we know $(d^q-\beta a)^2-(a^q-\beta d)^2\neq 0$ if and only if $A_4\neq \beta B_4$.

  Since $h(x)$ is a degree-$3$ permutation polynomial over $\mathbb{P}^{1}(\mathbb{F}_q)$, $\alpha A_1\neq \beta \alpha^q B_1$, $\alpha A_2=\beta \alpha^qB_2$, and $\frac{\beta \alpha^q B_3-\alpha A_3}{\alpha A_1-\beta \alpha^qB_1}$ is a nonsquare of $\mathbb{F}_q^{*}$.

  Since $\frac{c^q}{b}=\pm 1$, substituting $c^q=\pm b$ into $\alpha A_2-\beta \alpha^qB_2=0$, we have
  \begin{eqnarray}\label{eqn57}
      [b(\alpha^3+2\alpha)+c(2\alpha^2+1)](\alpha-\alpha^q)=0.
  \end{eqnarray}

  Since $\alpha\not\in \mathbb{F}_q$, we have $b(\alpha^3+2\alpha)+c(2\alpha^2+1)=0$. Moreover, $\alpha=\frac{cd^q-ab^q}{bd^q-ac^q}=\frac{c}{b}$, substituting $\frac{c}{b}$ into Equation (\ref{eqn57}), we get $0$. Hence this condition is automatically satisfied when $\alpha=\frac{c}{b}$. Since $A_1=\beta B_1$, $\alpha A_1-\beta \alpha^qB_1=\beta B_1(\alpha-\alpha^q)$. Moreover, $\alpha\not\in \mathbb{F}_q$, we know that $\alpha A_1\neq \beta \alpha^q B_1$.
  Hence, $h(x)$ is a permutation of $\mathbb{P}^{1}(\mathbb{F}_q)$ if and only if $\frac{\beta \alpha^q B_3-\alpha A_3}{\alpha A_1-\beta \alpha^q B_1}$ is a nonsquare of $\mathbb{F}_{q}^{*}$.

(3): $\alpha\in \mathbb{F}_q$, $\beta \not \in \mathbb{F}_q$. In this case, we can choose $\sigma=\frac{\beta x-\alpha\beta^q}{x-\alpha}$. It is easy to see that $\sigma^{-1}(x)=\frac{\alpha x-\alpha\beta^q}{x-\beta}$. Now we compute $h(x)=\rho\circ g\circ \sigma^{-1}$. First, we compute $g\circ \sigma^{-1}(x)$, that is
\begin{eqnarray}\label{eqn58}
    \begin{aligned}
        g\circ\sigma^{-1}(x)&=\frac{d^q\left(\frac{\alpha x-\alpha \beta^q}{x-\beta}\right)^3+c^q\left(\frac{\alpha x-\alpha \beta^q}{x-\beta}\right)^2+b^q\left(\frac{\alpha x-\alpha \beta^q}{x-\beta}\right)+a^q}{a\left(\frac{\alpha x-\alpha \beta^q}{x-\beta}\right)^3+b\left(\frac{\alpha x-\alpha \beta^q}{x-\beta}\right)^2+c\left(\frac{\alpha x-\alpha \beta^q}{x-\beta}\right)+d}\\
        &=\frac{d^q(\alpha^3x^3-\alpha^3\beta^{3q})+c^q(\alpha x-\alpha \beta^q)^2(x-\beta)+b^q(\alpha x-\alpha \beta^q)(x-\beta)^2+a^q(x^3-\beta^3)}{a(\alpha^3x^3-\alpha^3\beta^{3q})+b(\alpha x-\alpha \beta^q)^2(x-\beta)+c(\alpha x-\alpha \beta^q)(x-\beta)^2+d(x^3-\beta^3)}\\
        &=\frac{A_1x^3+A_2x^2+A_3x+A_4}{B_1x^3+B_2x^2+B_3x+B_4},
    \end{aligned}
\end{eqnarray}
where
\begin{eqnarray}\label{eqn59}
    A_1=d^q\alpha^3+c^q\alpha^2+b^q\alpha+a^q,
\end{eqnarray}
\begin{eqnarray}\label{eqn60}
    A_2=-(c^q(\beta\alpha^2+2\alpha^2\beta^q)+b^q(\alpha\beta^q+2\alpha\beta)),
\end{eqnarray}
\begin{eqnarray}\label{eqn61}
    A_3=c^q(2\alpha^2+\alpha^2\beta^{q-1})+b^q(2\alpha+\alpha\beta^2),
\end{eqnarray}
\begin{eqnarray}\label{eqn62}
    A_4=-(d^q\alpha^3\beta^{q-2}+c^q\alpha^2\beta^q+b^q\alpha\beta+a^q\beta^3),
\end{eqnarray}
\begin{eqnarray}\label{eqn63}
    B_1=a\alpha^3+b\alpha^2+c\alpha+d,
\end{eqnarray}
\begin{eqnarray}\label{eqn64}
    B_2=-(b(\beta\alpha^2+2\alpha^2\beta^q)+c(\alpha \beta^q+2\alpha\beta)),
\end{eqnarray}
\begin{eqnarray}\label{eqn65}
    B_3=b(2\alpha^2+\alpha^2\beta^{q-1})+c(2\alpha+\alpha\beta^2),
\end{eqnarray}
\begin{eqnarray}\label{eqn66}
    B_4=-(a\alpha^3\beta^{q-2}+b\alpha^2\beta^q+c\alpha\beta+d\beta^3).
\end{eqnarray}

This gives
\begin{eqnarray}\label{eqn67}
    h(x)=\frac{(\beta A_1-B_1)x^3+(\beta A_2-B_2)x^2+(\beta A_3-B_3)x+(\beta A_4-B_4)}{(A_1-\beta B_1)x^3+(A_2-\beta B_2)x^2+(A_3-\beta B_3)x+(A_4-\beta B_4)}.
\end{eqnarray}

Since $g(x)$ permutes $\mu_{q+1}$, $h(x)$ permutes $\mathbb{P}^{1}(\mathbb{F}_q)$. According to Lemma \ref{lem7}, we have $A_1=\beta B_1$, $A_2=\beta B_2$, $A_3=\beta B_3$, $A_4\neq \beta B_4$, $\beta A_1\neq B_1$, $\beta A_2=B_2$ and $\frac{B_3-\beta A_3}{\beta A_1-B_1}$ is a nonsquare in $\mathbb{F}_q$. Let us examine each condition individually to determine whether it holds.

Since $g(\alpha)=\beta$, we obtain $A_1=\beta B_1$. Moreover, $\beta\not\in \mathbb{F}_q$, we know $\beta A_1\neq B_1$. Since $b^q=\beta c$ and $c^q=\beta b$, we have $A_2=\beta B_2$ and $A_3=\beta B_3$. Now let us check  $A_4\neq \beta B_4$, which is equivalent to
\begin{eqnarray}\label{eqn68}
    \alpha^3\beta^{q-2}(d^q-\beta a)+\beta^3(a^q-\beta d)\neq 0.
    \end{eqnarray}

Assume that $\alpha^3\beta^{q-2}(d^q-\beta a)+\beta^3(a^q-\beta d)=0$. Multiplying both sides by $\beta^3$, we have
\begin{eqnarray}\label{eqn69}
    \alpha^3(d^q-\beta a)+\beta^6(a^q-\beta d)=0.
\end{eqnarray}

From Equation (\ref{eqn69}), we get $\beta^6=\frac{\alpha^3(\beta a-d^q)}{a^q-\beta d}=1$, which is a contradiction. Hence we have $A_4\neq \beta B_4$.
Finally, since $\beta A_2=B_2$, we have
\begin{eqnarray}\label{eqn70}
    (b-\alpha c)(\beta-\beta^q)=0.
\end{eqnarray}

Since $\alpha=\frac{b^q}{c^q}$ and $\alpha\in \mathbb{F}_q$ and $\beta\not\in \mathbb{F}_q$, $\alpha=\frac{b}{c}$. Hence this condition is automatically satisfied. Hence, $h(x)$ is a permutation of $\mathbb{P}^{1}(\mathbb{F}_q)$ if and only if $\frac{B_3-\beta A_3}{\beta A_1-B_1}$ is a nonsquare of $\mathbb{F}_q$.

(4): $\alpha\in \mathbb{F}_q$, $\beta\in \mathbb{F}_q$. In this case, we first choose an element $\gamma\in \mathbb{F}_{q^2}$ of order $q+1$. Then we choose $\sigma=\frac{\gamma x-\alpha\gamma^q}{x-\alpha}$, $\rho=\frac{\gamma x-\beta \gamma^q}{x-\beta}$. It is easy to know that $\sigma^{-1}(x)=\frac{\alpha x-\alpha\gamma^q}{x-\gamma}$. The composition $g\circ\sigma^{-1}(x)$ is similar to the previous case; namely, it has the form of $\frac{A_1x^3+A_2x^2+A_3x+A_4}{B_1x^3+B_2x^2+B_3x+B_4}$, where 
\begin{eqnarray}\label{eqn71}
    A_1=d^q\alpha^3+c^q\alpha^2+b^q\alpha+a^q,
\end{eqnarray}
\begin{eqnarray}\label{eqn72}
    A_2=-(c^q(\gamma\alpha^2+2\alpha^2\gamma^q)+b^q(\alpha\gamma^q+2\alpha\gamma)),
\end{eqnarray}
\begin{eqnarray}\label{eqn73}
    A_3=c^q(2\alpha^2+\alpha^2\gamma^{q-1})+b^q(2\alpha+\alpha\gamma^2),
\end{eqnarray}
\begin{eqnarray}\label{eqn74}
    A_4=-(d^q\alpha^3\gamma^{q-2}+c^q\alpha^2\gamma^q+b^q\alpha\gamma+a^q\gamma^3),
\end{eqnarray}
\begin{eqnarray}\label{eqn75}
    B_1=a\alpha^3+b\alpha^2+c\alpha+d,
\end{eqnarray}
\begin{eqnarray}\label{eqn76}
    B_2=-(b(\gamma\alpha^2+2\alpha^2\gamma^q)+c(\alpha\gamma^q+2\alpha\gamma)),
\end{eqnarray}
\begin{eqnarray}\label{eqn77}
    B_3=b(2\alpha^2+\alpha^2\gamma^{q-1})+c(2\alpha+\alpha\gamma^2),
\end{eqnarray}
\begin{eqnarray}\label{eqn78}
    B_4=-(a\alpha^3\gamma^{q-2}+b\alpha^2\gamma^q+c\alpha\gamma+d\gamma^3).
\end{eqnarray}

Now $h(x)$ is
 \begin{eqnarray}\label{eqn79}
 \begin{aligned}
      h(x)&=\frac{\gamma\left(\frac{A_1x^3+A_2x^2+A_3x+A_4}{B_1x^3+B_2x^2+B_3x+B_4}\right)-\beta \gamma^q}{\left(\frac{A_1x^3+A_2x^2+A_3x+A_4}{B_1x^3+B_2x^2+B_3x+B_4}\right)-\beta}\\
      &=\frac{\gamma(A_1x^3+A_2x^2+A_3x+A_4)-\beta \gamma^q(B_1x^3+B_2x^2+B_3x+B_4)}{A_1x^3+A_2x^2+A_3x+A_4-\beta(B_1x^3+B_2x^2+B_3x+B_4)}\\
      &=\frac{(\gamma A_1-\beta \gamma^q B_1)x^3+(\gamma A_2-\beta \gamma^q B_2)x^2+(\gamma A_3-\beta \gamma^q B_3)x+(\gamma A_4-\beta \gamma^q B_4)}{(A_1-\beta B_1)x^3+(A_2-\beta B_2)x^2+(A_3-\beta B_3)x+(A_4-\beta B_4)}.
 \end{aligned} 
 \end{eqnarray}
 Since $g(x)$ permutes $\mu_{q+1}$, $h(x)$ permutes $\mathbb{P}^{1}(\mathbb{F}_q)$. According to Lemma \ref{lem7}, we have $A_1=\beta B_1$, $A_2=\beta B_2$, $A_3=\beta B_3$, $A_4\neq \beta B_4$, $\gamma A_1\neq \beta \gamma^q B_1$, $\gamma A_2=\beta \gamma^qB_2$ and $\frac{\beta\gamma^q B_3-\gamma A_3}{\gamma A_1-\beta \gamma^qB_1}$ is a nonsquare in $\mathbb{F}_q$. Let us examine each condition individually to determine whether it holds. The discussion is similar to the previous case. The conclusion is $h(x)$ is a permutation over $\mathbb{P}^{1}(\mathbb{F}_q)$ if and only if $\frac{\beta\gamma^q B_3-\gamma A_3}{\gamma A_1-\beta\gamma^q B_1}$ is a nonsquare in $\mathbb{F}_q$. 
 
From the above discussion, we can obtain necessary and sufficient conditions for $g(x)$ to permute $\mu_{q+1}$ when the characteristic of $\mathbb{F}_q$ is $3$.

\begin{thm}\label{thm3}
    Let $g(x)=\frac{N(x)}{D(x)}=\frac{d^qx^3+c^qx^2+b^qx+a^q}{ax^3+bx^2+cx+d}\in \mathbb{F}_{q^2}(x)$. Assume that $\text{deg}(g)=3$. Let $\alpha=\frac{cd^q-ab^q}{bd^q-ac^q}$ and $\beta=\frac{c^q}{b}$. When the characteristic of $\mathbb{F}_q$ is $3$, the function $g(x)$ permutes $\mu_{q+1}$ if and only if $\alpha, \beta\in \mu_{q+1}$, Proposition \ref{pro2} holds,  and one of the following conditions is satisfied:

    (1) $\alpha\not\in \mathbb{F}_q$, $\beta\not\in \mathbb{F}_q$, and $\frac{B_3-\beta A_3}{\beta A_1-B_1}$ is a nonsquare of $\mathbb{F}_q^{*}$, where 
    $A_1=d^q\alpha^3+c^q\alpha^2+b^q\alpha+a^q$, $A_3=c^q(2\alpha^2+1)+b^q(2\alpha+\alpha^3)$, $B_1=a\alpha^3+b\alpha^2+c\alpha+d$, and $B_3=b(2\alpha^2+1)+c(2\alpha+\alpha^3)$;

    (2) $\alpha\not\in \mathbb{F}_q$, $\beta\in \mathbb{F}_q$, and $\frac{\beta\alpha^qB_3-\alpha A_3}{\alpha A_1-\beta\alpha^q B_1}$ is a nonsquare of $\mathbb{F}_q^{*}$, where $A_1,A_3,B_1,B_3$ are defined as in $(1)$;

    (3) $\alpha\in \mathbb{F}_q$, $\beta\not\in \mathbb{F}_q$, and $\frac{B_3-\beta A_3}{\beta A_1-B_1}$ is a nonsquare of $\mathbb{F}_q^{*}$, where $A_1=d^q\alpha^3+c^q\alpha^2+b^q\alpha+a^q$, $A_3=c^q(2\alpha^2+\alpha^2\beta^{q-1})+b^q(2\alpha+\alpha\beta^2)$, $B_1=a\alpha^3+b\alpha^2+c\alpha+d$, and $B_3=b(2\alpha^2+\alpha^2\beta^{q-1})+c(2\alpha+\alpha\beta^2)$;

    (4) $\alpha\in \mathbb{F}_q$, $\beta\in \mathbb{F}_q$. Let $\gamma\in \mathbb{F}_{q^2}$ be an element of order $q+1$. Then the expression $\frac{\beta\gamma^qB_3-\gamma A_3}{\gamma A_1-\beta\gamma^q B_1}$ is a nonsquare of $\mathbb{F}_q^{*}$, where $A_1=d^q\alpha^3+c^q\alpha^2+b^q\alpha+a^q$, $A_3=c^q(2\alpha^2+\alpha^2\gamma^{q-1})+b^q(2\alpha+\alpha\gamma^2)$, $B_1=a\alpha^3+b\alpha^2+c\alpha+d$, and $B_3=b(2\alpha^2+\alpha^2\gamma^{q-1})+c(2\alpha+\alpha\gamma^2)$.
\end{thm}

Now we show that when $\alpha, \beta \in \mathbb{F}_q$, the expression $\frac{\beta \gamma^q B_3-\gamma A_3}{\gamma A_1-\beta \gamma^q B_1}$ can be written solely in terms of the coefficients $a, b, c,$ and $d$. We take $\alpha = 1$ and $\beta = 1$ as an example; other cases can be handled similarly.

If $\alpha=\beta=1$, then from the expression of $\alpha$ and $\beta$, we obtain $c^q=b$ and $c(d^q-a)=b(d^q-a)$. From Proposition \ref{pro2}, we know that $d^q\neq a$. Since $c(d^q-a)=b(d^q-a)$, it follows that $b=c$. Together with $c^q = b$, we conclude that $b = c\in \mathbb{F}_q^{*}$.

Again, from Proposition \ref{pro2}, we know that $c^qd-a^qb=bd^q-ac^q$. Using $b=c\in \mathbb{F}_{q}^{*}$, we obtain $d^q+a^q=d+a$. Substituting these expressions into $A_1$, $A_3$, $B_1$, and $B_3$, we obtain

\begin{eqnarray}\label{eqn80}
    A_1=a+b+c+d,
\end{eqnarray}
\begin{eqnarray}\label{eqn81}
    A_3=b(\gamma^2+\gamma^{q-1}+1),
\end{eqnarray}
\begin{eqnarray}\label{eqn82}
    B_1=a+b+c+d,
\end{eqnarray}
\begin{eqnarray}\label{eqn83}
    B_3=b(\gamma^2+\gamma^{q-1}+1).
\end{eqnarray}

Plugging $A_1$, $A_3$, $B_1$, and $B_3$ into $\frac{\beta\gamma^q B_3-\gamma A_3}{\gamma A_1-\beta \gamma^q B_1}$, we obtain 
\begin{eqnarray}\label{eqn84}
    \frac{\beta\gamma^qB_3-\gamma A_3}{\gamma A_1-\beta \gamma^q B_1}=\frac{(\gamma^q-\gamma)b(\gamma^2+\gamma^{q-1}+1)}{(\gamma-\gamma^q)(a+b+c+d)}=-\frac{b(\gamma^2+\gamma^{q-1}+1)}{a+b+c+d}.
\end{eqnarray}

Since $\gamma^{q+1}=1$, we obtain $\gamma^2+\gamma^{q-1}+1=(\gamma-\gamma^q)^2$. Because $(\gamma-\gamma^q)^q=-(\gamma-\gamma^q)$, we have $(\gamma-\gamma^q)^{q-1}=-1$. Together with $(-\frac{b(\gamma^2+\gamma^{q-1}+1)}{a+b+c+d})^{\frac{q-1}{2}}=-1$, we deduce that $(-\frac{b}{a+b+c+d})^{\frac{q-1}{2}}=1$, which generalizes Lemma 3.8 in \cite{ding2025class}.

We summarize the above discussion as follows:

\begin{pro}\label{pro3}
    Let $g(x)=\frac{N(x)}{D(x)}=\frac{d^qx^3+c^qx^2+b^qx+a^q}{ax^3+bx^2+cx+d}\in \mathbb{F}_{q^2}(x)$. Assume that $\text{deg}(g)=3$, $\alpha=\frac{cd^q-ab^q}{bd^q-ac^q}\in \mathbb{F}_q\cap \mu_{q+1}$ and $\beta=\frac{c^q}{b}\in \mathbb{F}_q\cap \mu_{q+1}$. When the characteristic of $\mathbb{F}_q$ is $3$, the function $g(x)$ permutes $\mu_{q+1}$ if and only if Proposition \ref{pro2} holds and $-\frac{\beta b}{N(\alpha)}$ is a nonsquare of $\mathbb{F}_q^{*}$.
\end{pro}
\section{Applications}

In this section, we use our previous results to completely determine permutation polynomials that are induced by degree-$3$ rational functions. As examples, we consider $ax^{2q-1}+bx^q+cx^{2q-1}+dx^{q^2-q+1}$ and $ax^{3q}+bx^{2q+1}+cx^{q+2}+dx^3$ over $\mathbb{F}_{q^2}$.

First, let us look $h(x)=ax^{3q}+bx^{2q+1}+cx^{q+2}+dx^3$, which can also be written as $h(x)=x^3f(x^{q-1})$, where $f(x)=ax^3+bx^2+cx+d$. Using Corollary \ref{cor1} and Lemma \ref{lem5}, we obtain $f(x)$ is a permutation polynomial over $\mathbb{F}_{q^2}$ if and only if $\text{gcd}(3,q-1)=1$, $f(x)$ has no roots in $\mu_{q+1}$ and $g(x)=\frac{N(x)}{D(x)}=\frac{d^qx^3+c^qx^2+b^qx+a^q}{ax^3+bx^2+cx+d}$ permutes $\mu_{q+1}$. 

It is easy to see that if $ax^3+bx^2+cx+d=0$ has a root $\alpha$ in $\mu_{q+1}$, then $\alpha$ is also a root of $d^qx^3+c^qx^2+b^qx+a^q=0$. Hence, the degree of $\text{gcd}(N(x),D(x))$ is greater than $0$. Therefore, if $\text{deg}(g)=3$, then $f(x)$ has no roots in $\mu_{q+1}$. In this case, the condition can be omitted. Then, using Theorem \ref{thm1}, Theorem \ref{thm2}, Theorem \ref{thm3}, and Proposition \ref{pro3}, we explicitly determine the conditions under which $h(x)$ can be a permutation polynomial over $\mathbb{F}_{q^2}$.

If $\text{deg}(g)=2$, the degree of $\text{gcd}(N(x),D(x))$ is $1$, which shows $N(x)$ and $D(x)$ have only one common root $\alpha$, which is preserved by the function $x \mapsto x^{-q}$. Hence $\alpha\in \mu_{q+1}$, which implies that $f(x)$ has a root in $\mu_{q+1}$. In this case, $f(x)$ cannot be a permutation polynomial.

If $\text{deg}(g)=1$, the degree of $H(x)=\text{gcd}(N(x),D(x))$ is $2$. Note that $f(x)$ having no roots in $\mu_{q+1}$ is equivalent to $H(x)$ having no roots in $\mu_{q+1}$, since every root of $f(x)$ in $\mu_{q+1}$ must also be a root of $H(x)$. Hence, in this case, we should add the condition that $H(x)$ has no roots in $\mu_{q+1}$.

We summarize the above discussion as follows:

\begin{thm}\label{thm4}
    Let $q=p^m$, where $p$ is prime and $m$ is a positive integer. For $a,b,c,d\in \mathbb{F}_{q^2}$, the polynomial $h(x)=ax^{3q}+bx^{2q+1}+cx^{q+1}+dx^3$ is a permutation polynomial over $\mathbb{F}_{q^2}$ if and only if $\text{gcd}(3,q-1)=1$ and one of the following holds:

    (1) The degree of $g(x)=\frac{d^qx^3+c^qx^2+b^qx+a^q}{ax^3+bx^2+cx+d}$ is $3$, the characteristic of $\mathbb{F}_q$ is not $3$, and Theorem \ref{thm2} holds;

    (2) The rational function $g(x)=\frac{d^qx^3+c^qx^2+b^qx+a^q}{ax^3+bx^2+cx+d}$ is separable of degree $3$, the characteristic of $\mathbb{F}_q$ is $3$, and Theorem \ref{thm3} holds;

    (3) The rational function $g(x)=\frac{d^qx^3+c^qx^2+b^qx+a^q}{ax^3+bx^2+cx+d}$ is not separable, $d^{q+1}\neq a^{q+1}$;
    
    (4) The degree of $g(x)=\frac{d^qx^3+c^qx^2+b^qx+a^q}{ax^3+bx^2+cx+d}$ is $1$, and $H(x)=\text{gcd}(ax^3+bx^2+cx+d,d^qx^3+c^qx^2+b^qx+a^q)$ has no roots in $\mu_{q+1}$.
\end{thm}

The case $ax^{2q-1}+bx^q+cx+dx^{q^2-q+1}$ can be discussed similarly. We only present the results.

\begin{thm}\label{thm5}
    Let $q=p^m$, where $p$ is prime and $m$ is a positive integer. For $a,b,c,d\in \mathbb{F}_{q^2}$, the polynomial $h(x)=ax^{2q-1}+bx^{q}+cx+dx^{q^2-q+1}$ is a permutation polynomial over $\mathbb{F}_{q^2}$ if and only if $\text{gcd}(3,q-1)=1$ and one of the following holds:

    (1) The degree of $g(x)=\frac{d^qx^3+c^qx^2+b^qx+a^q}{ax^3+bx^2+cx+d}$ is $3$, the characteristic of $\mathbb{F}_q$ is not $3$, and Theorem \ref{thm2} holds;

     (2) The rational function $g(x)=\frac{d^qx^3+c^qx^2+b^qx+a^q}{ax^3+bx^2+cx+d}$ is separable of degree $3$, the characteristic of $\mathbb{F}_q$ is $3$, and Theorem \ref{thm3} holds;

    (3) The rational function $g(x)=\frac{d^qx^3+c^qx^2+b^qx+a^q}{ax^3+bx^2+cx+d}$ is not separable, $d^{q+1}\neq a^{q+1}$;

    (4) The degree of $g(x)=\frac{d^qx^3+c^qx^2+b^qx+a^q}{ax^3+bx^2+cx+d}$ is $1$, and $H(x)=\text{gcd}(ax^3+bx^2+cx+d,d^qx^3+c^qx^2+b^qx+a^q)$ has no roots in $\mu_{q+1}$.
    
\end{thm}
\section{Conclusion}

In this paper, by using the classification results of low-degree permutation rational functions, we explicitly determine all  rational functions of the form $\frac{d^qx^3+c^qx^2+b^qx+a^q}{ax^3+bx^2+cx+d}$, where $a,b,c,d\in \mathbb{F}_{q^2}$, which induce a permutation on $\mu_{q+1}$. Using these results, we can explicitly determine many permutation polynomials which are induced from such functions. It would be interesting to explicitly determine more permutation rational functions over $\mu_{q+1}$ in order to construct additional permutation polynomials over $\mathbb{F}_{q^2}$.

\section{Acknowledgement}
The authors thank the editor and anonymous referees for their efforts in improving the readability of this paper.

\bibliographystyle{IEEEtran}
\bibliography{ffa-refs}  









\end{document}